\documentclass[10pt,twoside,english,reqno,a4paper]{amsart}
\usepackage{listings,graphicx,amsmath,varioref,amscd,amssymb,color,bm,stmaryrd,amsthm,amsfonts,graphics,geometry,latexsym,pgf,pst-all} 
\theoremstyle{plain}
\usepackage{esint}
\usepackage{amsthm}
\theoremstyle{plain}
\newtheorem{theorem}{Theorem}[section]

\newtheorem{lemma}[theorem]{Lemma}

\usepackage{geometry}
\usepackage[colorlinks=false]{hyperref}

\theoremstyle{definition}
\newtheorem{defin}[theorem]{Definition}

\newtheorem{remark}[theorem]{Remark}
\newtheorem{example}{Example}

\theoremstyle{remark}

\usepackage[light,condensed,math]{iwona}
\usepackage[T1]{fontenc}

\numberwithin{equation}{section}
\def\dis{\displaystyle}
\def\supp{\text{\text{supp}}}

\DeclareMathOperator{\R}{\mathbb{R}}

\newcommand{\car}[1]{\raise1pt\hbox{$\chi$}_{#1}}

\newcommand{\DM }{\mathcal{DM}^\infty }

\begin{document}

\title[...]{Regularizing effect of absorption terms in singular problems}

\author{Francescantonio Oliva}
\address{Francescantonio Oliva\\ Dipartimento di Scienze di Base e Applicate per l' Ingegneria, Sapienza Universit\`a di Roma\\Via Scarpa 16, 00161 Roma, Italy}
\email{francesco.oliva@sbai.uniroma1.it}


%\date{..., ...}

\begin{abstract}
	
	\noindent We prove existence of solutions to problems whose model is  
	$$\begin{cases}
	\displaystyle   -\Delta_p u + u^q = \frac{f}{u^\gamma} & \text{in}\ \Omega,
	\\
	u\ge0 &\text{in}\ \Omega,
	\\
	u=0 &\text{on}\ \partial\Omega,
	\end{cases}$$
	where $\Omega$ is an open bounded subset of $\mathbb{R}^N$ ($N\ge2$), $\Delta_p u$ is the $p$-laplacian operator for $1\le p <N$, $q>0$, $\gamma\ge 0$ and $f$ is a nonnegative function in $L^m(\Omega)$ for some $m\ge1$. In particular we analyze the regularizing effect produced by the absorption term in order to infer the existence of finite energy solutions in case $\gamma\le 1$. We also study uniqueness of these solutions as well as examples which show the optimality of the results. Finally, we find local $W^{1,p}$-solutions in case $\gamma>1$.
\end{abstract}
\maketitle
\tableofcontents
\section{Introduction}
 The aim of this work is the study of the following boundary value problem 
  \begin{equation}
  \begin{cases}
  \displaystyle   -\Delta_p u + g(u) = h(u)f & \text{in}\ \Omega,\\
  u\ge 0  & \text{in}\ \Omega,\\
  u=0 &\text{on}\ \partial\Omega,
  \end{cases}
  \label{pbintro}
  \end{equation}
  where, for $1\le p < N$, the $p$-laplacian operator is $\Delta_p u:= \operatorname{div}(|\nabla u |^{p-2}\nabla u)$. Here $\Omega\subset\mathbb{R}^N \ (N\ge2)$ is open and bounded (with Lipschitz boundary if $p=1$), $f$ is nonnegative and it belongs to $L^m(\Omega)$ for some $m\ge1$ while  $g(s)$ is continuous, $g(0)=0$ and, as $s\to\infty$, could act as $s^q$ with $q\ge -1$.
  Finally $h$ is continuous, it possibly blows up at the origin and it is bounded at infinity. We should think to $h(s)$ as a non-monotone function which grows at most as $s^{-\gamma}$ near zero and as $s^{-\theta}$ at infinity with $\gamma,\theta\ge 0$. We highlight that the case of continuous, bounded and non-monotone functions $g,h$ is covered by the above assumptions.\\ Our main goal is the existence of finite energy solutions to \eqref{pbintro} (i.e. $u\in W^{1,p}_0(\Omega)$ if $p>1$ and $u\in BV(\Omega)$ if $p=1$); in particular we are interested in understanding the role of the absorption term $g$ in order to produce a regularizing effect in terms of Sobolev regularity of the solutions to \eqref{pbintro} in presence of a possibly singular $h$ as well as the regularizing effect given by $h$ itself when it goes to zero fast enough at infinity.
  \\  
  \\Problem \eqref{pbintro} when $p>1$ and $g\equiv 0$ has been widely studied; if $p=2$, $h(s) = s^{-\gamma}$ ($\gamma>0$) and $f$ is a regular function, existence of classical solutions comes from \cite{crt, lm, s}. Only later, in \cite{bo}, the authors prove existence of a distributional solution in case of a Lebesgue datum $f$ and remarked the regularizing effect given by the right hand side of \eqref{pbintro} when, once again, $p=2$ and $h(s) = s^{-\gamma}$ ($\gamma>0$): namely the solution always lies in a smaller Sobolev space when $0<\gamma\le 1$ compared to the one of case $\gamma=0$. Moreover if $\gamma=1$ the solution is always in $H^1_0(\Omega)$ even if $f$ is just an $L^1$-function as one can formally deduce by taking $u$ itself as test function in \eqref{pbintro} while if $\gamma>1$ the solution belongs only locally to $H^1(\Omega)$ and the boundary datum is given as a suitable power of the solution having zero Sobolev trace. Then, in \cite{op2}, a general $h$ (as above) is considered and various results are proved depending on $\gamma$, $\theta$, and $f$ in order to have finite energy solutions. Among other things, one of the results concerning finite energy solutions can be summarized as follows: let $0\le \gamma\le 1$ then $u \in H^1_0(\Omega)$ if $0\le \theta<1$ and $f\in L^{\left(\frac{2^*}{1-\theta}\right) '}(\Omega),$  or if $\theta\ge 1$ and $f$ is just an $L^1$-function. We also underline that, if $\theta=0$, we recover the classical regularity results. In this framework a natural question is how the presence of $g$ can affect the problem in order to deduce $H^1_0$-solutions when $\theta<1$ and $f\in L^m(\Omega)$ with $1<m<\left(\frac{2^*}{1-\theta}\right)'$. For various features of this kind of singular problems we refer to the following works and references therein \cite{bgh, cmss,cst, car, cc, dc,ddo,do, diaz, far,gmm,gmm2,gk,gcs,op,orpe,sz}.
  \\ \\
  For what concerns the regularizing effect given by the absorption term $g$, the first contribution comes from \cite{bgv}. Here the authors, when $h(s)=1$  and $f\in L^1(\Omega)$, deal with existence of solutions to problem \eqref{pbintro}; in particular they prove that larger is $q$ better is the Sobolev regularity of the solution.\\ In the same direction we also recall \cite{cir} where it is shown that if $f\in L^m(\Omega)$ with $m>1$ and 
  \begin{equation}\label{qintrocirmi}
  	q\ge \frac{1}{m-1},
  \end{equation}
  then the solution to \eqref{pbintro} has always finite energy. More regularizing effect of this kind are discussed in \cite{arbo,arbo2,amm,bc,croce} and their references.
  \\ \\In this paper we deal with the regularizing effect given by both the absorption lower order term and the rate to which $h(s)$ goes to zero as $s\to \infty$. In Theorem \ref{teo_p>1} below we prove the existence of solutions in $W^{1,p}_0(\Omega)$ if $h$ mildly blows up at the origin (i.e. $0\le \gamma\le 1$); in particular we reach $W^{1,p}_0$-solutions if $\theta< 1$ and
 \begin{equation}\label{qintro}
 q\ge \frac{1-m\theta}{m-1},
 \end{equation}   
 or, independently on $q$, if $\theta\ge 1$. Some remarks are in order:  the result of the above mentioned theorem is sharp as shown in Example \ref{examplesharp} below, moreover we also observe that as $\theta$ goes to zero we recover \eqref{qintrocirmi}, which fits with the result of \cite{cir} and finally if $\theta\ge \frac{1}{m}$ we just do not need an unbounded absorption term anymore in order to have finite energy solutions. We also recall that for $p=2$ some partial results for problems as in \eqref{pbintro} were proved in \cite{dco2} with some limitations on the choice of $g$ due to the need of applying the maximum principle; we underline that here no use of the maximum principle is in order to manage the possibly singular function $h$ and, hence, the function $g$ can be way more general. We conclude the discussion for $p>1$  highlighting that we also tackle \eqref{pbintro} in case $\gamma>1$, which is a quite different situation: in this case only locally finite solutions are expected to exist and just a power of the solution lies in $W^{1,p}_0(\Omega)$ for sufficiently regular $f$ as one can formally deduce by multiplying \eqref{pbintro} with $u^\gamma$. Therefore, the absorption term, by increasing the Lebesgue summability of the solution, allows to deduce the existence of local finite energy solutions even if the datum is not regular (see Theorem \ref{teo_p>1strong} below).
  \\ \\ When $p=1$, we refer to \cite{ABCM} where for the first time the authors proposed the Anzellotti theory (see \cite{anz}) to represent $\Delta_1=\operatorname{div}( |D u|^{-1} D u)$. Here they take advantage of the pairing theory $(z, Du)$ between a gradient of a function in $BV(\Omega)$ and a bounded vector field $z$ with $||z||_{L^\infty(\Omega)^N}\le 1$ having distributional divergence as a Radon measure with bounded total variation; in this way $z$ plays the role of $|D u|^{-1} D u$.
  \\ \\
  At the best of our knowledge the literature concerning problems as in \eqref{pbintro} with $p=1$ is limited. In absence of the absorption term and when $h$ is equal to one, then existence of a $BV$-solution is proved when $f\in L^N(\Omega)$, provided its norm is small enough (see \cite{CT,MST1}). When $f$ lies just in $L^1(\Omega)$ we refer to \cite{MST2}, where it is proved the existence of a suitable notion of solution to \eqref{pbintro}; for instance, they proved existence of solutions having just their truncations in $BV(\Omega)$. 
  In presence of an absorption term type we mainly refer to \cite{ls,ls2} where the authors deal with the regularizing effect given by a first order term when $h\equiv1$. Furthermore, in the very recent work \cite{dgs}, it is proved existence of a solution when $h(s)=s^{-\gamma}$ ($\gamma<1$) and $f\in L^N(\Omega)$ without requiring any smallness condition on the norm. In \cite{dgop} the authors obtain existence (and uniqueness when expected) of local $BV$-solutions to \eqref{pbintro} when $h$ is not necessarily monotone and possibly blows up with any $\gamma\ge0$ and $f \in L^N(\Omega)$.   For more features about problems involving the $1$-Laplace operator we refer to \cite{ADS,D,gmp,K,KS,MP}.
  \\ \\For what concerns our work we essentially prove (Theorem \ref{teo_p>1} below) that if $\gamma \le 1$, $f$ is just in $L^m(\Omega)$ ($m\ge 1$) and condition \eqref{qintro} is satisfied, then the presence of the absorption term gives rise to existence of a $BV$-solution to \eqref{pbintro}; we also highlight that, even if $\theta$ is equal to zero, we do not require any smallness assumption on the norm of $f$. Hence if $q$ is large enough we always have $BV$-solutions independently on the size of $f$ and its summability. One of the keys is that thanks to the absorption term, we are always in position to show that the pairing $(z, Du)$ is well defined. We also highlight that an additional difficulty is that our solutions are not expected to be bounded as we are not assuming regular data. This fact appears in Example \ref{exampleunbounded}, where an unbounded solution to a problem as \eqref{pbintro} is explicitly shown to exist. Finally we also show the existence of locally finite energy solutions in case $\gamma>1$ (Theorem \ref{teo_p>1strong} below). More precisely we prove that if $q$ is large enough then there exist local $BV$-solutions for any $f\in L^m(\Omega)$ with $m>1$.
  \\ \\
  Furthermore we deal with uniqueness of solutions for both cases $p>1$ and $p=1$; in Theorems \ref{unip>1}, \ref{uni_p=1} we prove uniqueness in the class of finite energy solutions if some integrability conditions on the absorption term are satisfied and $g,h$ are monotone functions. Finally, if $p>1$,  we show that even if the effect of the absorption is not strong enough in order to have finite energy solutions then anyway it gives rise to a regularization on the Sobolev regularity of the solutions.
  \\ \\ The paper is organized as follows. In Section \ref{sec:prel} we give some preliminaries, we extend to our framework the definition of the pairing $(z, Du)$, and we recall a Gauss-Green type formula.
  In Section \ref{sec:ass} we present the problem and the statement of the main existence results. In Section \ref{sec:apriori} we introduce the approximation scheme and deduce the main estimates needed in Section \ref{sec:exi} which is devoted to the proofs of existence results both in case $p>1$ and $p=1$. In Section \ref{sec:uni} we prove uniqueness of solutions when expected while in Section \ref{sec:example} we give examples showing the sharpness of our existence results and a more general result in case of infinite energy solutions ($p>1$) is also given. In the same section we also present a case where bounded solutions exist even in presence of rough data. Finally, in Section \ref{sec:strong}, we briefly deal with the case of a strong singularity, namely $h$ blows up faster at the origin. 
           
\subsection{Notations} 
\label{not}
For a given function $v$ we denote by $v^+=\max(v,0)$ and by $v^-= -\min (v,0)$. Moreover $\chi_{E}$ denotes the characteristic function of a set $E$. For a fixed $k>0$, we define the truncation functions $T_{k}:\R\to\R$ and $G_{k}:\R\to\R$ as follows
\begin{align*}
	T_k(s):=&\max (-k,\min (s,k)),\\
	G_k(s):=&(|s|-k)^+ \operatorname{sign}(s).
\end{align*}
We will also use the following functions
\begin{align}\label{Vdelta}
	\displaystyle
	V_{\delta,k}(s):=
	\begin{cases}
		1 \ \ &s\le k, \\
		\displaystyle\frac{k+\delta-s}{\delta} \ \ &k <s< k+\delta, \\
		0 \ \ &s\ge k+\delta,
	\end{cases}
\end{align}
and 
\begin{equation}\label{Sdelta}
	S_{\delta,k}(s):=1-V_{\delta,k}(s).	
\end{equation}
We denote by $\mathcal H^{N-1}(E)$ the $(N - 1)$-dimensional Hausdorff measure of a set $E$ while $|E|$ stands for its $N$-dimensional  Lebesgue measure.\\
For the entire paper $\Omega$ is an open bounded subset of $\R^N$ ($N\ge 1$) with Lipschitz boundary if $p=1$ while $\mathcal{M}(\Omega)$ is the usual space of Radon measures with finite total variation over $\Omega$. By $W^{1,p}_0(\Omega)$ we mean the Sobolev space with zero trace and by $L^{N,\infty}(\Omega)$ the classical Lorentz space. We refer to a Lebesgue space with respect to a Radon measure $\mu$ as $L^q(\Omega,\mu)$. We also denote by 
$$\DM(\Omega):=\{ z\in L^\infty(\Omega;\R^N) : \operatorname{div}z \in \mathcal{M}(\Omega) \},$$
and by $\DM_{\rm loc}(\Omega)$ its local version, namely the space of bounded vector field $z$ with $\operatorname{div}z \in \mathcal{M}(\omega)$ for every $\omega \subset\subset \Omega$.   
We also recall that
$$BV(\Omega):=\{ u\in L^1(\Omega) : Du \in \mathcal{M}(\Omega, \R^N) \}.$$ 
We underline that the $BV(\Omega)$ space endowed with the norm  
$$ ||u||_{BV(\Omega)}=\int_\Omega |u|\, + \int_\Omega|Du|\,,$$
or with
$$\displaystyle ||u||_{BV(\Omega)}=\int_{\partial\Omega}
|u|\, d\mathcal H^{N-1}+ \int_\Omega|Du|,$$
is a Banach space. We denote by $BV_{\rm loc}(\Omega)$ the space of functions in $BV(\omega)$ for every $\omega \subset\subset\Omega$.
For more properties regarding $BV$ spaces we refer to \cite{AFP}.\\
We explicitly remark that, if no otherwise specified, we will denote by $C$ several positive constants whose value may change from line to line and, sometimes, on the same line. These values will only depend on the data but they will never depend on the indexes of the sequences we will introduce.

\section{Preliminary facts}
\label{sec:prel}

In order to deal with the $1$-laplacian operator we briefly recall the theory of $L^\infty$-divergence-measure vector fields (see \cite{anz} and \cite{CF}).
First we recall that if $z\in \DM(\Omega)$ then it can be proved that $\operatorname{div}z $ is an absolutely continuous measure with respect to $\mathcal H^{N-1}$. 
\\Moreover, as in \cite{anz}, we define the following distribution $(z,Dv): C^1_c(\Omega)\to \mathbb{R}$ 
\begin{equation}\label{dist1}
\langle(z,Dv),\varphi\rangle:=-\int_\Omega v^*\varphi\operatorname{div}z-\int_\Omega
vz\cdot\nabla\varphi,\quad \varphi\in C_c^1(\Omega),
\end{equation}
where $v^* $ always denotes the precise representative of $v$. Following the idea in \cite{anz}, in \cite{MST2} and \cite{C} the authors prove that $(z, Dv)$ is well posed if $z\in \DM(\Omega)$ and $v\in BV(\Omega)\cap L^\infty(\Omega)$ since one can show that $v^*\in L^\infty(\Omega,\operatorname{div}z)$. Moreover in \cite{dgs} the authors show that \eqref{dist1} is well posed if $z\in \DM_{\rm loc}(\Omega)$ and $v\in BV_{\rm loc}(\Omega)\cap L^1_{\rm loc}(\Omega, \operatorname{div}z)$.
\\Moreover, reasoning as in \cite{dgs}, one deduces that, once $v^*\in L^1_{\rm loc}(\Omega, \operatorname{div}z)$, $(z, Dv)$ is a Radon measure with local finite total variation satisfying
\begin{equation*}\label{finitetotal}
|\langle   (z, Dv), \varphi\rangle| \le ||\varphi||_{L^{\infty}(U) } ||z||_{L^\infty(U)^N} \int_{U} |Dv|\,,
\end{equation*}
for all open set $U \subset \Omega$ and for all $\varphi\in C_c^1(U)$, from which one also deduces that
\begin{equation*}\label{finitetotal1}
\left| \int_B (z, Dv) \right|  \le  \int_B \left|(z, Dv)\right| \le  ||z||_{L^\infty(U)^N} \int_{B} |Dv|\,,
\end{equation*}
for all Borel sets $B$ and for all open sets $U$ such that $B\subset U \subset \Omega$.\\
We recall that every $z \in \mathcal{DM}^{\infty}(\Omega)$ has a weak trace on $\partial \Omega$ of the normal component of  $z$ which is denoted by
$[z, \nu]$, where $\nu(x)$ is the outward normal unit vector defined for $\mathcal H^{N-1}$-almost every $x\in\partial\Omega$ (see \cite{anz}). Moreover, it satisfies
\begin{equation*}\label{des1}
||[z,\nu]||_{L^\infty(\partial\Omega)}\le ||z||_{L^\infty(\Omega)^N},
\end{equation*}
and also that if $z \in \mathcal{DM}^{\infty}(\Omega)$ and $v\in BV(\Omega)\cap L^\infty(\Omega)$, then
\begin{equation}\label{des2}
v[z,\nu]=[vz,\nu]
\end{equation}
holds (see \cite{C}).\\
Finally, in \cite{dgs}, the authors prove that if $z\in \DM_{\rm loc}(\Omega)$ and $v\in BV(\Omega)\cap L^\infty(\Omega)$ such that $v^*\in L^1(\Omega,\operatorname{div}z)$ then $vz\in \DM(\Omega)$ and a weak trace can be defined as well as a Gauss-Green formula which we recall for the sake of completeness.
\begin{lemma}
	Let $z\in \DM_{\rm loc}(\Omega)$ and let $v\in BV(\Omega)\cap L^\infty(\Omega)$ such that $v^*\in L^1(\Omega,\operatorname{div}z)$ then 
	\begin{equation*}\label{green}
	\int_{\Omega} v^* \operatorname{div}z + \int_{\Omega} (z, Dv) = \int_{\partial \Omega} [vz, \nu] \ d\mathcal H^{N-1}.
	\end{equation*}	
\end{lemma}
Hence, since we are interested in proving that the distribution \eqref{dist1} is well defined, in the next lemma we prove that, under suitable assumptions on $\operatorname{div}z$, it holds $v^*\in L^1_{\rm loc}(\Omega, \operatorname{div}z)$. 
\begin{lemma}\label{lempairing}
	Let $v\in BV_{\rm loc}(\Omega)$ be nonnegative and let $\xi :\mathbb{R}\to \mathbb{R}$ be a continuous nonnegative function with $\xi(0)=0$ such that $\xi(v)\in BV_{\rm loc}(\Omega)$. Moreover let $\chi_{\{v>0\}}\in BV_{\rm loc}(\Omega)$, $z\in \mathcal{DM}^{\infty}_{\rm loc}(\Omega)$ and let
	\begin{equation}\label{1}
	-\psi(v)^* \operatorname{div}z = \sigma \ \ \text{ as measures in $\Omega$, }
	\end{equation}
	where $ \psi(s) = 1$ or $\psi(s)= \chi_{\{s>0\}}$ and where $\sigma\in L^1_{\rm loc}(\Omega)$ such that $\sigma \xi(v) \in L^1_{\rm loc}(\Omega)$.
	Then $\xi(v)^*\in L^1_{\rm loc}(\Omega,\operatorname{div}z)$. 
\end{lemma}
\begin{proof}
		Let us observe that, since  $\xi(v)\in BV_{\rm loc}(\Omega)$, $\xi(v)^*$ is measurable with respect to $\psi(v)^* \operatorname{div}z$. Then, for any $\omega \subset \subset \Omega$,  one has that
	\begin{equation}\label{intfinito}
		-\int_\omega  \xi(v)^*\psi(v)^* \operatorname{div}z = \int_\omega \sigma \xi(v),
	\end{equation} 
	and, by the assumptions on $\sigma$, the right hand side of \eqref{intfinito} is finite and if $\psi(s)=1$ the proof is concluded. Hence, assuming $\psi(s)= \chi_{\{s>0\}}$, from \eqref{intfinito} one has that $\xi(v)^*\chi^*_{\{v>0\}}\in L^1_{\rm loc}(\Omega,\operatorname{div}z)$. Now observe that
	$$\xi(v)^* = \xi(v)^*\chi^*_{\{v>0\}} + \xi(v)^*\chi^*_{\{v=0\}},$$
	for $\mathcal{H}^{N-1}$-almost every $x\in \Omega$. Finally we note that 
	$$\xi(v)^*\chi^*_{\{v=0\}}\le \xi(v)^*\chi^*_{\{v>0\}},$$
	for $\mathcal{H}^{N-1}$-almost every $x\in \Omega$ and this concludes the proof of the Lemma.
\end{proof}
\begin{remark}
We remark that under the assumptions of Lemma \ref{lempairing} if one supposes $\xi(v)\in BV_{\rm loc}(\Omega)$ and $\sigma\in L^1_{\rm loc}(\Omega)$ such that $\sigma \xi(v)\in L^1(\Omega)$ then one obtains that $\xi(v)^*\in L^1(\Omega,\operatorname{div}z)$.
\end{remark}
We close this section with a lemma which is a slight improvement of a result already contained in \cite{dgop,dgs} and which consists in a regularity result for the vector field $z$.

\begin{lemma}\label{lemmal1}
	Let $0\le \sigma \in L^{1}_{\rm{loc}}(\Omega)$, let $\tau\in L^1(\Omega)$ and let $z\in \mathcal{D}\mathcal{M}^\infty_{\rm{loc}}(\Omega)$ with $||z||_{L^\infty(\Omega)^N}\le 1$ such that 
	\begin{equation}\label{lemma_distr*}
	-\operatorname{div}z +\tau= \sigma \text{  as measures in $\Omega$,  }
	\end{equation}
	then 
	\begin{equation}\label{l1}
	\operatorname{div}z \in L^1(\Omega).
	\end{equation}	
\end{lemma}
\begin{proof}
	First we prove that the set of admissible test functions in \eqref{lemma_distr*} can be enlarged. If one takes $0\le v\in W^{1,1}_0(\Omega)\cap L^\infty(\Omega)$ then there exists a sequence of nonnegative functions $v_{\eta,n}\in C^1_c(\Omega)$ such that (with an abuse of notation, $v_n$ will be the almost everywhere limit of $v_{\eta,n}$ as $\eta \to 0$)
	\begin{equation}\label{proptest}
	\begin{cases}
	v_{\eta,n} \stackrel{\eta  \to 0}{\to} v_{n} \stackrel{n  \to \infty}{\to} v \ \ \ \text{in } W^{1,1}_0(\Omega) \text{ and } *\text{-weakly in } L^\infty(\Omega), \\
	\supp v_n \subset \subset \Omega: 0\le v_n\le v \ \ \ \text{for all } n\in \mathbb{N}.
	\end{cases}
	\end{equation}
	A good example of such $v_{\eta,n}$ is given by $\rho_\eta \ast (v \wedge \phi_n)$ ($v \wedge \phi_n:= \inf (v,\phi_n)$) where $\rho_\eta$ is a sequence of smooth mollifier while $\phi_n$ is a sequence of nonnegative functions in $C^1_c(\Omega)$ which converges to $v$ in $W^{1,1}_0(\Omega)$.
	If one takes $v_{\eta,n}$ as test function in \eqref{lemma_distr*} then
	\begin{equation*}\label{contest}
	\int_\Omega z \cdot \nabla v_{\eta,n} + \int_\Omega \tau v_{\eta,n} = \int_\Omega \sigma v_{\eta,n}.
	\end{equation*}
	We first pass to the limit as $\eta$ goes to zero. Recalling \eqref{proptest}, for the left hand side we pass to the limit since $z\in L^\infty(\Omega)^N$ and $\tau\in L^1(\Omega)$
	For the right hand side we observe that, for $\eta$ small enough,  $\supp v_{\eta,n} \subset\subset \Omega$ and then we can pass the limit by Lebesgue Theorem since $\sigma\in L^1_{\rm loc}(\Omega)$. Hence one has
	\begin{equation}\label{contest2}
	\int_\Omega z \cdot \nabla v_{n} + \int_\Omega \tau v_{n} = \int_\Omega \sigma v_{n}.
	\end{equation}
	Now we need to pass to the limit with respect to $n$; for the left hand side we can reason as before as $\eta\to 0$.
	For the right hand side of \eqref{contest2} in order to apply the Lebesgue theorem, one needs that $\sigma v \in L^1(\Omega)$.
	Hence one has
	\begin{equation*}
	\int_\Omega \sigma v_n = \int_\Omega z\cdot \nabla v_n + \int_\Omega \tau v_n  \leq ||z||_{L^\infty(\Omega)^N}\int_\Omega  |\nabla v_n| + ||v||_{L^\infty(\Omega)}||\tau||_{L^1(\Omega)}\le C,
	\end{equation*}
	which, after an application of the Fatou Lemma with respect to $n$, gives $\sigma v \in L^1(\Omega)$.  
	This means we can pass to the limit also in the right hand side of \eqref{contest2} deducing 
	\begin{equation}\label{appe1}
	\int_\Omega z \cdot \nabla v + \int_\Omega \tau v = \int_\Omega \sigma v,
	\end{equation}	 
	for all nonnegative $v\in W^{1,1}_0(\Omega)\cap L^\infty(\Omega)$.\\
	Now it follows from Lemma $5.5$ of \cite{anz} that if $\tilde{v} \in W^{1,1}(\Omega)\cap L^\infty(\Omega)$ then there exists $w_n\in
	W^{1, 1}(\Omega)\cap C(\Omega)$ having $w_n|_{\partial\Omega}=\tilde{v}|_{\partial\Omega}$, $\displaystyle\int_\Omega|\nabla
	w_n|\,dx\le\displaystyle\int_{\partial\Omega}\tilde{v}\,d\mathcal H^{N-1}+\frac1n\,,$ and such that $w_n$ tends to $0$ in $\Omega$. 
	Hence $|v-w_n|\in W_0^{1,1}(\Omega)$ can be chosen as a test function in \eqref{appe1} taking to
	\begin{align*}
	\int_\Omega \sigma|\tilde{v}-w_n| &= \int_\Omega z\cdot\nabla|\tilde{v}-w_n| + \int_\Omega \tau |\tilde{v}-w_n| 
	\\
	&\le ||z||_{L^\infty(\Omega)^N}\int_\Omega|\nabla \tilde{v}| + ||z||_{L^\infty(\Omega)^N} \int_\Omega|\nabla w_n| +\int_\Omega \tau |\tilde{v}-w_n|
	\\
	&\le \int_\Omega|\nabla \tilde{v}|+\int_{\partial\Omega}\tilde{v}\,d\mathcal H^{N-1}+\frac1n +\int_\Omega \tau |\tilde{v}-w_n|.
	\end{align*}
	An application of the Fatou Lemma gives
	\begin{equation}\label{appl1}
	\int_\Omega \sigma|\tilde{v}|\le \int_\Omega|\nabla \tilde{v}|+\int_{\partial\Omega}\tilde{v}\,d\mathcal H^{N-1} + \int_\Omega \tau |\tilde{v}|.
	\end{equation}
	Now if one takes $\tilde{v}\equiv 1$ in \eqref{appl1} then one gets $\sigma\in L^1(\Omega)$, which also implies that $\operatorname{div}z \in L^1(\Omega)$.
\end{proof}

\section{Main assumptions and a first existence result}
\label{sec:ass}

Let us consider the following problem
\begin{equation}
\label{pb1}
\begin{cases}
\dis -\Delta_p u + g(u)= h(u)f & \text{in}\;\Omega,\\
u=0 & \text{on}\;\partial\Omega,
\end{cases}
\end{equation}
where $1\le p<N$, $\Omega$ is an open bounded subset of $\mathbb{R}^N$ with Lipschitz boundary if $p=1$, $f\in L^m(\Omega)$ with $m\ge 1$ and $h:[0,\infty)\to [0,\infty]$ is a continuous and possibly singular function with $h(0)\not=0$ which it is finite outside the origin and such that
\begin{equation}
\exists \gamma\ge 0, \underline{C},\underline{s}>0: h(s)\le \frac{\underline{C}}{s^\gamma} \ \text{for all } s\le \underline{s}, 
\label{h1}\tag{h1}
\end{equation}
and 
\begin{equation}
\exists \theta\ge 0, \overline{C}>0, \overline{s}> \underline{s}: h(s)\le \frac{\overline{C}}{s^\theta} \ \text{for all } s\ge  \overline{s}. 
\label{h2}\tag{h2}
\end{equation}
We underline that both $\gamma$ and $\theta$ are allowed to be zero so that a continuous and bounded function is an admissible choice.
The absorption term $g:[0,\infty)\to [0,\infty)$ is a continuous function such that $g(0)=0$ and, if $\theta<1$, the following growth condition at infinity holds
\begin{equation}\label{g1}\tag{g1}
\exists q \ge \frac{1-m\theta}{m-1}, \  \exists\nu, s_1>0 : g(s)\geq \nu s^{q} \ \text{for all}\;s\geq s_1.
\end{equation}
We explicitly observe that $g$ can be any nonnegative continuous function with $g(0)=0$ if $\theta \ge 1$. Indeed, as we will see, if this is the case then the regularizing effect given by the right hand side of \eqref{pb1} is already sufficient to have solutions in the energy space (i.e. solutions in $W^{1,p}_0(\Omega)$ if $p>1$ and $BV(\Omega)$ if $p=1$). If $0\le \theta <1$, i.e. the function $h(s)$ does not go to zero quickly enough as $s\to \infty$, we need to impose that $g(s)$ "grows" at least as a (not necessarily positive) power as $s\to\infty$.   
 The way the $p$-laplacian operator is understood is quite different between cases $1<p<N$ and $p=1$, which means the need of two notions of distributional solution to problem \eqref{pb1}.  
\\The first one concerns the case $1<p<N$.
\begin{defin}
	\label{weakdefp>1}
	Let $1<p<N$ then a nonnegative $u\in W^{1,1}_0(\Omega)$ such that $|\nabla u|^{p-1} \in L^1(\Omega)$ is a solution to problem \eqref{pb1} if $g(u), h(u)f \in L^1_{\rm loc}(\Omega)$ and it holds
	\begin{align}\label{def_distr}
	\int_{\Omega} |\nabla u|^{p-2}\nabla u \cdot \nabla \varphi + \int_{\Omega} g(u)\varphi = \int_{\Omega} h(u)f\varphi, \ \ \ \forall \varphi\in C^1_c(\Omega).
	\end{align}
\end{defin}
Then we give the notion of solution when $p=1$; here for a solution we mean a $BV$-function and we need to introduce a vector field which formally plays the role of $|Du|^{-1}Du$.
\begin{defin}
\label{weakdefpositive}
			Let $p=1$ then a nonnegative $u\in BV(\Omega)$ such that $\chi_{\{u>0\}} \in BV_{\rm loc}(\Omega)$ is a solution to problem \eqref{pb1} if $g(u), h(u)f \in L^1_{\rm loc}(\Omega)$ and if there exists $z\in \mathcal{D}\mathcal{M}^\infty_{\rm loc}(\Omega)$ with $||z||_{L^\infty(\Omega)^N}\le 1$ such that
			\begin{align}
				&-\psi^*(u) \operatorname{div}z + g(u) = h(u)f \ \ \text{as measures in }\Omega, \label{def_distrp=1}
				\\
				&\text{where } \psi(s) = \begin{cases}1 \ \ \ &\text{if } h(0)<\infty, \\ \chi_{\{s>0\}} \ \ \ &\text{if } h(0)=\infty, \end{cases} \nonumber
				\\
				&(z,Du)=|Du| \label{def_zp=1} \ \ \ \ \text{as measures in } \Omega,
				\\
				&T_k(u(x)) + [T_k(u)z,\nu] (x)=0 \label{def_bordop=1}\ \ \ \text{for  $\mathcal{H}^{N-1}$-a.e. } x \in \partial\Omega \ \text{and for every} \ k>0.
			\end{align}
\end{defin}
\begin{remark}\label{rempos}
	We highlight some features of Definition \ref{weakdefpositive}. First of all we observe that the presence of a characteristic function when $h(0)=\infty$ seems to be quite natural; in some sense, we read the distribution formulation where the solution is strictly positive. Indeed one can observe that $h(u)f\in L^1_{\rm loc}(\Omega)$ means that
	 \begin{equation}\label{uposrem}
	 \{u=0\}\subset \{f=0\},
	 \end{equation}
	and then 
	$$\int_\Omega h(u)f\varphi = \int_\Omega h(u)f\chi_{\{u>0\}}\varphi.$$
	We also observe that if $f>0$ almost everywhere in $\Omega$ then \eqref{uposrem} guarantees that $u>0$ almost everywhere in $\Omega$ which implies that \eqref{def_distrp=1} reads as 
	\begin{equation}\label{distrfpos}
	-\operatorname{div}z + g(u) = h(u)f \ \text{as measures in }\Omega,
	\end{equation}	
	as for the case $h(0)<\infty$ and this will be essential in order to prove uniqueness Theorem \ref{uni_p=1} below. Indeed if \eqref{distrfpos} holds and one also has $g(u)\in L^1(\Omega)$ then it can be deduced that $\operatorname{div}z \in L^1(\Omega)$ (see Lemma \ref{lemmal1}), i.e. $z\in \DM(\Omega)$. Moreover, once $z\in \DM(\Omega)$, recalling \eqref{des2}, then \eqref{def_bordop=1} easily takes to 
	$$	u(x)( 1 + [z,\nu] (x))=0 \ \ \  \text{for  $\mathcal{H}^{N-1}$-a.e. } x \in \partial\Omega.$$
\end{remark}
Now we are ready to state our existence theorem. 
\begin{theorem}\label{teo_p>1}
	Let $1\le p<N$, let $h$ satisfy \eqref{h1} with $\gamma\le1$, \eqref{h2}, and suppose that one of the following assumptions hold:
	\begin{itemize}
		\item[i)] $\theta\ge1 \text{ and } f\in L^1(\Omega)$;

		\item[ii)] $\theta<1,f\in L^{m}(\Omega)$ with $m>1$, and $g$ satisfies \eqref{g1}. 
	\end{itemize}	
	Then there exists a solution $u$ to problem \eqref{pb1} such that $g(u)u\in L^1(\Omega)$. Moreover if $1<p<N$ then $u$ belongs to $W^{1,p}_0(\Omega)$.
\end{theorem}
\begin{remark}
	For the sake of presentation, we separately present the result in the case $\gamma>1$ in Section \ref{sec:strong} since, as already discussed in the introduction, the solutions have just locally finite energy and the notion of solution is different (see Definitions \ref{distributional} and \ref{weakdefpositivestrong} below).\\
	We also highlight that Theorem \ref{teo_p>1} in case $p>1$ is sharp as shown in Example \ref{examplesharp} below: if $q<  \frac{1-m\theta}{m-1}$ then, in general, is not possible to expect finite energy solutions. Moreover in Example \ref{exampleunbounded}, for $p=1$, we find an unbounded solution to the problem for a rough datum and for any $q$: this means that the regularizing effect of $g$ is not strong enough to deduce the boundedness of the solution. Finally we remark that between cases i) and ii) of Theorem \ref{teo_p>1} there is continuity in the summability of the datum. Indeed in case ii) the condition on $q$ implies that
	$$m\ge \frac{q+1}{q+\theta},$$
	which, as $\theta$ goes to $1$, gives $m\ge 1$.
\end{remark}

\section{A \emph{priori} estimates}
\label{sec:apriori}

In order to prove Theorem \ref{teo_p>1} we work by approximation: we truncate the possibly singular function $h$ obtaining an approximated solution which will take to a solution $u_p$ in case $p>1$ and then, moving $p$ to one, one deduces the existence of a solution also in this limit case. The goal of this section is the introduction of the scheme of approximation and the proof of estimates which need to be independent of the level of truncation and of $p$. We underline that the solution $u$ found for $p=1$ is the one constructed from the scheme of approximation which takes to $u_p$; this is fundamental since, in general, there is no uniqueness of solutions when $1<p<N$ (see Theorem \ref{unip>1} below for a uniqueness result). Let $p>1$ and let us introduce the following scheme of approximation.     
\begin{equation}
\label{pbpn}
\begin{cases}
\dis -\Delta_p u_{n,k} + g_{k}(u_{n,k})= h_n(u_{n,k})f_n & \text{in}\;\Omega,\\
u_{n,k}=0 & \text{on}\;\partial\Omega,
\end{cases}
\end{equation}
where $g_{k}(s)=T_k(g(s))$ for $s\ge 0$ and $g_{k}(s)=0$ for $s<0$, $h_n(s)= T_n(h(s))$, $h_n(s)= h_n(0)$ for $s<0$ and $f_n=T_n(f)$. The existence of a weak solution $u_{n,k} \in W^{1,p}_0(\Omega)$ is guaranteed by \cite{ll} and, by standard Stampacchia's type theory, $u_{n,k}\in L^\infty(\Omega)$. Moreover taking $u_{n,k}^-$ as a test function in the weak formulation of \eqref{pbpn} one has that $u_{n,k}$ is nonnegative. Indeed
$$-\int_\Omega |\nabla u_{n,k}^-|^p + \int_\Omega g_{k}(u_{n,k}) u_{n,k}^- =  \int_{\Omega} h_n(u_{n,k})f_n u_{n,k}^- \ge 0,$$
which gives $u_{n,k}^- \equiv 0$ on $\Omega$ namely $u_{n,k}\ge 0$ almost everywhere in $\Omega$.
\\Now, for $t>0$, we take $G_t(u_{n,k})$ as a test function in the weak formulation of \eqref{pbpn} and dropping the nonnegative absorption term, one deduces 
$$\int_\Omega |\nabla G_t(u_{n,k})|^p \le \int h_n(u_{n,k})f_n G_t(u_{n,k}) \le \sup_{s\in [t,\infty)}[h(s)]  \int f_n G_t(u_{n,k}),$$
which is known to imply that $||u_{n,k}||_{L^\infty(\Omega)}\le C$ for some positive constant $C$ that is independent of $k$. 
\\We also observe that, taking $u_{n,k}$ as a test function in the weak formulation of \eqref{pbpn}, one deduces that $u_{n,k}$ is bounded in $W^{1,p}_0(\Omega)$ with respect to $k$ and it converges, as $k\to\infty$ and up to subsequences, to some function which we denote by $u_n$, solution to
\begin{equation}
\label{pbpn2}
\begin{cases}
\dis -\Delta_p u_{n} + g(u_{n})= h_n(u_{n})f_n & \text{in}\;\Omega,\\
u_{n}=0 & \text{on}\;\partial\Omega.
\end{cases}
\end{equation}
Hence our aim in this section is proving some estimates for $u_n$ pointing out the dependence on parameters $n,p$.
\begin{lemma}\label{lemmapriori0}
Let $h$ satisfy \eqref{h1} with $\gamma\le 1$, \eqref{h2}, and suppose that one of the following assumptions hold:
	\begin{itemize}
	\item[i)] $\theta\ge1 \text{ and } f\in L^1(\Omega)$;
	
	\item[ii)] $\theta<1,f\in L^{m}(\Omega)$ with $m>1$, and $g$ satisfies \eqref{g1}. 
\end{itemize}	
If $u_n$ is a solution to \eqref{pbpn2} then 
\begin{equation}\label{stimapriori0}
	||u_n||^p_{W^{1,p}_0(\Omega)} + ||g(u_n)u_n||_{L^{1}(\Omega)}\le C.
\end{equation}
   Moreover
   	\begin{equation}\label{lemstimal1}
    \int_{\Omega}h_n(u_n)f_n\varphi\le C+||\varphi||^p_{W^{1,p}_0(\Omega)} + C||\varphi||_{L^\infty(\Omega)},
    \end{equation}
    for all $\varphi\in W^{1,p}_0(\Omega)\cap L^\infty(\Omega)$. In both cases $C$ is a positive constant independent of $n$ and $p$.
\end{lemma}
\begin{proof}
	Let us take $u_n$ as a test function in the weak formulation of \eqref{pbpn2}, obtaining
	\begin{equation}\label{priori01}
	\begin{aligned}
	\int_{\Omega} |\nabla u_n|^p + \int_{\Omega}g(u_n)u_n &= \int_{\Omega} h_n(u_{n})f_n u_n \le \underline{C}\int_{\{u_n< \underline{s}\}} f_n u_n^{1-\gamma}
	\\
	& + \int_{\{\underline{s}\le u_n \le \overline{s}\}} h_n(u_{n})f_n u_n + \overline{C}\int_{\{u_n> \overline{s}\}} f_n u_n^{1-\theta}.
	\end{aligned}
	\end{equation}
	Now in case i) we estimate \eqref{priori01} as follows
	\begin{equation*}
	\begin{aligned}
	\int_{\Omega} |\nabla u_n|^p + \int_{\Omega}g(u_n)u_n &= \int_{\Omega} h_n(u_{n})f_n u_n \le  \left(\underline{C}\underline{s}^{1-\gamma} + \max_{s\in [\underline{s},\overline{s}]} [h(s)s] + \overline{C}\overline{s}^{1-\theta}\right) ||f||_{L^1(\Omega)},
	\end{aligned}
	\end{equation*}
	which gives \eqref{stimapriori0}. In case ii) we apply the Young inequality on the right hand side of \eqref{priori01} which gives 
	\begin{equation}\label{priori1}
	\begin{aligned}
	\int_{\Omega} |\nabla u_n|^p + \int_{\Omega}g(u_n)u_n &\le  \left(\underline{C}\underline{s}^{1-\gamma} + \max_{s\in [\underline{s}, \overline{s}]} [h(s)s]\right)||f||_{L^1(\Omega)} + C_\varepsilon||f||^m_{L^m(\Omega)} 
	\\
	&+\varepsilon\int_{\{u_n> \overline{s}\}} u_n^{\frac{(1-\theta)m}{m-1}}.
	\end{aligned}
	\end{equation}
	Without loss of generality we suppose that $s_1\le  \overline{s}$. Hence, if $q= \frac{1-m\theta}{m-1}$, recalling \eqref{g1} and fixing $\varepsilon$ small enough then the proof simply follows. Otherwise, from \eqref{g1} and applying the Young inequality with indexes $\left(\frac{(q+1)(m-1)}{(1-\theta)m},\frac{(q+1)(m-1)}{(q+1)(m-1)-(1-\theta)m} \right)$ on the last term on the right hand side of \eqref{priori1}, one gets 
\begin{equation*}\label{priori2}
\begin{aligned}
\int_{\Omega} |\nabla u_n|^p + \nu \int_{\{u_n > s_1\}}u_n^{q+1} &\le \left(\underline{C}\underline{s}^{1-\gamma} + \max_{s\in [\underline{s}, \overline{s}]} [h(s)s]\right)||f||_{L^1(\Omega)}+ C_\varepsilon||f||^m_{L^m(\Omega)} 
\\
&+\frac{ \varepsilon(1-\theta)m}{(q+1)(m-1)} \int_{\{u_n> s_1\}}u_n^{q+1} + \frac{\varepsilon[(q+1)(m-1)-(1-\theta)m]}{(q+1)(m-1)}|\Omega|,
\end{aligned}
\end{equation*}
and then
	\begin{equation*}
	\begin{aligned}
	\int_{\Omega} |\nabla u_n|^p + \left(\nu-\frac{ \varepsilon(1-\theta)m}{(q+1)(m-1)}\right)\int_{\{u_n> s_1\}}u_n^{q+1} &\le \left(\underline{C}\underline{s}^{1-\gamma} + \max_{s\in [\underline{s}, \overline{s}]} [h(s)s]\right)||f||_{L^1(\Omega)}
	\\
	& + C_\varepsilon||f||^m_{L^m(\Omega)} + \frac{\varepsilon[(q+1)(m-1)-(1-\theta)m]}{(q+1)(m-1)}|\Omega|,
	\end{aligned}
	\end{equation*}
	which, fixing $\varepsilon$ small enough and coupling with \eqref{priori1}, implies \eqref{stimapriori0}.
	Once that \eqref{stimapriori0} holds we take $\varphi\in W^{1,p}_0(\Omega)\cap L^\infty(\Omega)$ as a test function in the weak formulation of \eqref{pbpn2} and we estimate its right hand side as  
	\begin{equation*}\label{stimal1}
	\begin{aligned}
	\int_{\Omega}h_n(u_n)f_n\varphi &\le \frac{p-1}{p}\int_{\Omega}|\nabla u_n|^p + \frac{1}{p}\int_{\Omega}|\nabla \varphi|^p + \int_{\Omega}g(u_n)\varphi
	\\
	&\le ||u_n||^p_{W^{1,p}_0(\Omega)}+ ||\varphi||^p_{W^{1,p}_0(\Omega)} + ||g(u_n)||_{L^1(\Omega)}||\varphi||_{L^\infty(\Omega)}.
	\end{aligned}
	\end{equation*} 
	Therefore \eqref{lemstimal1} follows by \eqref{stimapriori0}. This concludes the proof.
\end{proof}
\begin{remark}	
	We explicitly highlight that in case $\theta \ge 1$ the regularizing effect given by the right hand side of \eqref{pbpn2} is so strong that we are not actually taking advantage of the absorption term at the Sobolev level. Instead its effect is clearly evident in the Lebesgue regularity of the solution. 
\end{remark}

\section{Proof of Theorem \ref{teo_p>1}}
\label{sec:exi}

\subsection{The case $p>1$}
\label{sec:p>1}

In this section we prove Theorem \ref{teo_p>1} in case $p>1$ through the scheme of approximation \eqref{pbpn2} and the estimates obtained in Section \ref{sec:apriori}.
\begin{proof}[Proof of Theorem \ref{teo_p>1} in case $p>1$]
	Let $u_n$ be a solution to \eqref{pbpn2}, then it follows from Lemma \ref{lemmapriori0} that it is bounded in $W^{1,p}_0(\Omega)$ with respect to $n$.
	Hence there exists  a function $u_p\in W^{1,p}_0(\Omega)$ such that $u_n$, up to subsequences, converges to $u_p$ in $L^{r}(\Omega)$ for all $r<\frac{pN}{N-p}$ and weakly in $W^{1,p}_0(\Omega)$. Lemma \ref{lemmapriori0} also gives that $h_n(u_n)f_n$ is bounded in $L^1_{\rm loc}(\Omega)$ and, clearly, $g(u_n)$ is bounded in $L^1(\Omega)$ with respect to $n$. Hence one can apply Theorem $2.1$ of \cite{bm} which gives that $\nabla u_n$ converges to $\nabla u_p$ almost everywhere in $\Omega$. We also underline that an application of the Fatou Lemma with respect to $n$ in \eqref{stimapriori0} allows to deduce that $g(u_p)u_p\in L^1(\Omega)$.
	\\Now we prove that $u_p$ satisfies \eqref{def_distr} by passing to the limit in $n$ every term in the weak formulation of \eqref{pbpn2}. We can easily pass to the limit the first term in \eqref{pbpn2} with respect to $n$; hence we focus on the absorption term $g$, which we show to be equi-integrable. Indeed if we test \eqref{pbpn2} with $\displaystyle S_{\eta,k}(u_n)$ (defined in \eqref{Sdelta}) where $\eta,k>0$ and we deduce
	\begin{equation*}
		\int_{\Omega} |\nabla u_n|^{p} S'_{\eta,k}(u_n) + \int_{\Omega} g(u_n) S_{\eta,k}(u_n) \le \sup_{s\in [k,\infty)}[h(s)]\int_{\Omega} f_n S_{\eta,k}(u_n),
	\end{equation*}
	which, observing that the first term on the left hand side is nonnegative and taking the limit with respect to $\eta \to 0$, implies 
	\begin{equation}\label{equi}
	\int_{\{u_n\ge k\}} g(u_n) \le \sup_{s\in [k,\infty)}[h(s)]\int_{{\{u_n\ge k\}} } f_n,
	\end{equation}	
	which, since $f_n$ converges to $f$ in $L^m(\Omega)$, easily implies that $g(u_n)$ is equi-integrable and so it converges to $g(u_p)$ in $L^1(\Omega)$. This is sufficient to pass to the limit in the second term of the weak formulation of \eqref{pbpn2}.
	\\Then in order to conclude the proof of the theorem we just need to treat the right hand side of \eqref{pbpn2}. If $h(0)<\infty$ then we can simply pass to the limit through the Lebesgue Theorem and the proof is done. This means that, without loss of generality, we assume $h(0)=\infty$ for the rest of the proof. From now we consider a nonnegative $\varphi\in W^{1,p}_0(\Omega)\cap L^\infty(\Omega)$. An application of the Fatou Lemma in \eqref{lemstimal1} with respect to $n$ gives   
	\begin{equation}\label{lemstimal1fatou}
	\int_{\Omega}h(u_p)f\varphi\le C,
	\end{equation}
	where $C$ does not depend on $n$. Moreover, from \eqref{lemstimal1fatou} one deduces that, up to a set of zero Lebesgue measure,
	\begin{equation}\label{upos}
	\{u_p = 0\} \subset \{f = 0\}. 
	\end{equation}
 Now, for $\delta >0$, we split the right hand side of \eqref{pbpn2} as 
\begin{equation}\label{rhs}
\int_{\Omega}h_n(u_n)f_n\varphi = \int_{\{u_n\le \delta\}}h_n(u_n)f_n\varphi + \int_{\{u_n> \delta\}}h_n(u_n)f_n\varphi,
\end{equation}
   and we pass to the limit first as $n\to \infty$ and then as $\delta\to 0$. We remark that we need to choose $\delta\not\in \{\eta: |\{u_p=\eta \}|>0\}$ which is at most a countable set.
	For the second term of \eqref{rhs} we have 
	$$h_n(u_n)f_n\varphi\chi_{\{u_n> \delta\}} \le \sup_{s\in [\delta,\infty)}[h(s)]\ f\varphi \in L^1(\Omega),$$
	which permits to apply the Lebesgue Theorem with respect to $n$. Hence one has
	\begin{equation*}\label{rhs2}
	\lim_{n\to \infty}\int_{\{u_n> \delta\}}h_n(u_n)f_n\varphi= \int_{\{u_p> \delta\}}h(u_p)f\varphi.
	\end{equation*}
	Moreover it follows by \eqref{lemstimal1fatou} that 
	$$h(u_p)f\varphi\chi_{\{u_p> \delta\}} \le h(u_p)f\varphi \in L^1(\Omega),$$
	and then, once again by the Lebesgue Theorem, one gets
	\begin{equation}\label{rhs21}
\lim_{\delta\to 0}\lim_{n\to \infty}\int_{\{u_n> \delta\}}h_n(u_n)f_n\varphi= \int_{\{u_p> 0\}}h(u_p)f\varphi.
\end{equation}	
Now in order to get rid of the first term of the right hand side of \eqref{rhs}, we take $V_{\delta,\delta}(u_n)\varphi$ ($V_{\delta,\delta}(s)$ is defined in \eqref{Vdelta}) as test function in the weak formulation of \eqref{pbpn2}, obtaining (recall $V'_{\delta,\delta}(s)\le 0$ for $s\ge 0$)
\begin{equation*}\label{limn1}
\begin{aligned}
\int_{\{u_n\le \delta\}}h_n(u_n)f_n\varphi\le \int_{\Omega}h_n(u_n)f_nV_{\delta,\delta}(u_n)\varphi \le \int_{\Omega}|\nabla u_n|^{p-2}\nabla u_n\cdot \nabla \varphi V_{\delta,\delta}(u_n) + \int_{\Omega}g(u_n)V_{\delta,\delta}(u_n)\varphi,
\end{aligned}
\end{equation*}
and then, as $n\to\infty$, by the weak convergence and by the Lebesgue Theorem, one gets
\begin{equation*}\label{limn2}
\begin{aligned}
\limsup_{n\to\infty}\int_{\{u_n\le \delta\}}h(u_n)f_n\varphi\le \int_{\Omega}|\nabla u_p|^{p-2}\nabla u_p\cdot \nabla \varphi V_{\delta,\delta}(u_p) + \int_{\Omega}g(u_p)V_{\delta,\delta}(u_p)\varphi.
\end{aligned}
\end{equation*}
We can now take the limit in $\delta\to 0$, obtaining (recall $g(0)=0$)
\begin{equation}\label{limn3}
\begin{aligned}
\lim_{\delta\to 0}\limsup_{n\to\infty}\int_{\{u_n\le \delta\}}h(u_n)f_n\varphi\le \int_{\{u_p=0\}}|\nabla u_p|^{p-2}\nabla u_p\cdot \nabla \varphi + \int_{\{u_p=0\}}g(u_p)\varphi = 0.
\end{aligned}
\end{equation}
Hence \eqref{rhs21}, \eqref{limn3} and \eqref{upos} imply that 
\begin{equation*}
\lim_{n\to \infty}\int_{\Omega}h_n(u_n)f_n\varphi = \int_{\{u_p>0\}}h(u_p)f\varphi = \int_{\Omega}h(u_p)f\varphi.
\end{equation*}
Therefore we have proved that there exists $u_p\in W^{1,p}_0(\Omega)$ such that, for all nonnegative $\varphi\in W^{1,p}_0(\Omega)\cap L^\infty(\Omega)$, it holds
\begin{equation}\label{eqp>1}
\int_{\Omega}|\nabla u_p|^{p-2}\nabla u_p \cdot \nabla \varphi + \int_\Omega g(u_p)\varphi =  \int_{\Omega}h(u_p)f\varphi,
\end{equation}
whence one deduces \eqref{def_distr}. This concludes the proof.
\end{proof}

\begin{remark}\label{set}
	From the previous proof we get the existence of a solution $u_p\in W^{1,p}_0(\Omega)$ to \eqref{pb1} where $p>1$ such that $g(u_p)u_p\in L^1(\Omega)$ and satisfying the weak formulation for a class of test functions larger than the one requested in \eqref{def_distr}, namely $\varphi\in W^{1,p}_0(\Omega)\cap L^\infty(\Omega)$. This is consistent with the proof of the uniqueness Theorem \ref{unip>1} below where we show that if $u_p\in W^{1,p}_0(\Omega)$ satisfies \eqref{def_distr} and the absorption term $g(u_p)$ is integrable in $\Omega$ then the set of test functions can be enlarged through a density argument.  
	\\Finally we also underline that the request $g(0)=0$ is only employed in \eqref{limn3}. The same conclusion can be deduced if, for instance, $f>0$ almost everywhere in $\Omega$ and $g(0)> 0$. Indeed, if this is the case, \eqref{upos} gives that $u_p>0$ almost everywhere in $\Omega$ and then the right hand side of \eqref{limn3} is still zero. 
\end{remark}

\subsection{The case $p=1$}
\label{sec:p1}

Here we prove Theorem \ref{teo_p>1} when $p=1$. We underline that, from here on, $u_p$ is the solution found in the previous section; i.e. $u_p$ solves
\begin{equation}
\label{pbp}
\begin{cases}
\dis -\Delta_p u_p + g(u_p)= h(u_p)f & \text{in}\;\Omega,\\
u_p=0 & \text{on}\;\partial\Omega,
\end{cases}
\end{equation}
where $1<p<N$ and it is obtained through the scheme of approximation \eqref{pbpn2} whence we will deduce most of estimates; finally we also highlight that $u_p$ solves \eqref{eqp>1} which is a slight more general formulation than the one given in \eqref{def_distr} (see also Remark \ref{set}). Hence our goal becomes moving $p\to 1$ in problem \eqref{pbp}. 
\begin{proof}[Proof of Theorem \ref{teo_p>1} in case $p=1$]
Let $u_p$ be the solution to \eqref{pbp} found in the proof of Theorem \ref{teo_p>1} in case $p>1$. First we observe that from the Young inequality and from the weak lower semicontinuity of the norm with respect to $n$ in \eqref{stimapriori0}, one gets
\begin{equation}
\begin{aligned}\label{bv}
\int_{\Omega} |\nabla u_p| \le \int_{\Omega} |\nabla u_p|^p  + \frac{p-1}{p} |\Omega|\le C,		   
\end{aligned}	
\end{equation}			
for some constant $C$ which does not depend on $p$. This means that $u_p$ is bounded in $BV(\Omega)$ with respect to $p$ (recall that the Sobolev trace of $u_p$ is zero) and therefore one can deduce the existence of a function $u\in BV(\Omega)$ such that $u_p$ (once again up to subsequences) converges to $u$ in $L^{r}(\Omega)$ with $r<\frac{N}{N-1}$ and $\nabla u_p$ converges to $Du$ $*$-weakly as measures as $p$ tends to $1$. Now that we have our candidate to be a solution we need to prove that \eqref{def_distrp=1}, \eqref{def_zp=1}, and \eqref{def_bordop=1} hold. We proceed by steps.

\medskip

	{\it Existence of the field $z$.}
  
From \eqref{bv} and from the H\"older inequality one has that, for $1\le q<\frac{p}{p-1}$, 
	\begin{align}\label{stimaz}
		|||\nabla u_p|^{p-2}\nabla u_p||_{L^q(\Omega)^N} \le \left(\int_\Omega |\nabla u_p|^p\right)^{\frac{p-1}{p}}|\Omega|^{\frac{1}{q}-\frac{p-1}{p}}\leq C^{\frac{p-1}{p}}|\Omega|^{\frac{1}{q}-\frac {p-1}{p}},		
	\end{align}
	whence one deduces the existence of a vector field $z_q\in L^q(\Omega)^N$ such that $|\nabla u_p|^{p-2}\nabla u_p$ converges weakly to $z_q$ in $L^q(\Omega)^N$;  then a standard  diagonal argument assures the existence of a unique vector field $z$, defined independently of $q$, such that $|\nabla u_p|^{p-2}\nabla u_p$ converges weakly to $z$ in $L^q(\Omega)^N$ for any $q<\infty$.	Moreover as $p\to 1$ then by weakly lower semicontinuity in \eqref{stimaz} one has $||z||_{L^q(\Omega)^N}\le |\Omega|^{\frac1q}$ for any $q<\infty$ and if $q\to \infty$ we also have $||z||_{L^\infty(\Omega)^N}\le 1$.

	\bigskip 

	{\it $u$ satisfies the distributional formulation \eqref{def_distrp=1}.}\\
	First of all we observe that if we let $n\to\infty$ in \eqref{equi} we still have that $g(u_p)$ is equi-integrable with respect to $p$, and so $g(u_p)$ converges to $g(u)$ in $L^1(\Omega)$. Moreover, applying the Fatou Lemma to \eqref{stimapriori0} first in $n$ and then in $p$, one gets $g(u)u\in L^1(\Omega)$. Now if $h(0)<\infty$ then we can simply pass to the limit the weak formulation \eqref{pbp} which also gives that $z\in \DM(\Omega)$. Hence without loss of generality we assume $h(0)=\infty$.  From here, if not explicitly stated, $\varphi$ will be a nonnegative function in $C^1_{c}(\Omega)$. We take $\varphi$ itself as a test in the weak formulation of \eqref{pbp} and by the Fatou Lemma, as $p\to 1$, one has
	\begin{equation}\label{dmloc}
		\int_{\Omega} z\cdot \nabla \varphi + \int_\Omega g(u)\varphi\ge \int_\Omega h(u)f\varphi \ge 0.
	\end{equation} 	
	Since the left hand side of the previous is finite we have that $h(u)f \in L^1_{\rm loc}(\Omega)$, which also implies (up to a set of zero Lebesgue measure)
	\begin{equation}\label{u0f0}
	 \{u=0\}\subset\{f=0\}.
	\end{equation}
	Moreover from \eqref{dmloc} we also deduce that $z\in \mathcal{DM}^\infty_{\rm loc}(\Omega)$. 
	Now we test the weak formulation of \eqref{pbp} with  $S_{\delta,\delta}(u_p)\varphi$ ($S_{\delta,\delta}$ is defined in \eqref{Sdelta}) and using also the Young inequality we obtain
	\begin{equation}
		\begin{aligned}\label{minore2}
			&\int_{\Omega} |\nabla S_{\delta,\delta}(u_p)|\varphi + \int_{\Omega}|\nabla u_p|^{p-2} \nabla u_p\cdot \nabla \varphi S_{\delta,\delta}(u_p) + \int_\Omega g(u_p)S_{\delta,\delta}(u_p)\varphi
			\\ 
			&\le \frac{1}{p}\int_{\Omega} |\nabla u_p|^p S'_{\delta,\delta}(u_p)\varphi + \frac{p-1}{p}\int_{\Omega} S'_{\delta,\delta}(u_p)\varphi + \int_{\Omega}|\nabla u_p|^{p-2} \nabla u_p\cdot \nabla \varphi S_{\delta,\delta}(u_p)
			\\ 
			&+ \int_\Omega g(u_p)S_{\delta,\delta}(u_p)\varphi \le \frac{p-1}{p}\int_{\Omega}  S'_{\delta,\delta}(u_p)\varphi +  \int_{\Omega}h(u_p)fS_{\delta,\delta}(u_p)\varphi.
		\end{aligned} 	
	\end{equation}
	We observe that $|\nabla S_{\delta,\delta}(u_p)| \le\frac1\delta |\nabla u_p|$ and then \eqref{bv} implies that $S_{\delta,\delta}(u_p)$ is bounded in $BV(\Omega)$ with respect to $p$. Hence, as $p$ tends to $1$, one can apply lower semicontinuity for the first term on the left hand side of \eqref{minore2} while for the second and third term we observe that $S_{\delta,\delta}(u_p)$ converges to $S_{\delta,\delta}(u)$ $*$-weakly in $L^\infty(\Omega)$. For the right hand side we have that the first term goes to zero as $S'_{\delta,\delta}$ is bounded and for the second term we can apply the Lebesgue Theorem ($S_{\delta,\delta}\le 1$). Thus one gets
	\begin{align*}
		\int_{\Omega} \varphi|D S_{\delta,\delta}(u)| + \int_{\Omega}z\cdot \nabla \varphi S_{\delta,\delta}(u) + \int_{\Omega}g(u)S_{\delta,\delta}(u)\varphi \le \int_{\Omega}h(u)fS_{\delta,\delta}(u)\varphi.
	\end{align*}
	Now we easily see that the second, the third and the fourth terms are bounded in $\delta$, which means that $S_{\delta,\delta}(u)$ is bounded in $BV_{\rm loc}(\Omega)$ with respect to $\delta$. Now taking $\delta\to 0$ and by lower semicontinuity for the first term on the left hand side we get	
	$$
	\int_{\Omega} \varphi|D \chi_{\{u>0\}}| + \int_{\Omega}z\cdot \nabla \varphi \chi_{\{u>0\}} + \int_{\Omega}g(u)\varphi \le \int_{\Omega}h(u)f\chi_{\{u>0\}}\varphi\stackrel{\eqref{u0f0}}{=}\int_{\Omega}h(u)f\varphi.
	$$				
	and hence (recall that $||z||_{L^\infty(\Omega)^N}\le 1$) one has that
	\begin{align}\label{minore4bis}
		-\int_{\Omega}\chi^*_{\{u>0\}}\varphi\operatorname{div}z  + \int_{\Omega}g(u)\varphi \le \int_{\Omega}h(u)f\varphi.
	\end{align} 
	Now set in \eqref{dmloc} $\varphi = (\rho_\epsilon*\chi_{\{u>0\}})\phi$ where $0\le \phi \in C^1_c(\Omega)$ and $\rho_\epsilon$ is a mollifier. Then as $\epsilon\to 0$ it follows that
	\begin{equation}\label{maggiore}
		-\int_{\Omega}\chi^*_{\{u>0\}}\phi\operatorname{div}z + \int_{\Omega} g(u) \phi \ge \int_{\Omega}h(u)f\chi_{\{u>0\}}\phi =	\int_{\Omega}h(u)f\phi \ \ \ \forall\phi \in C^1_c(\Omega), \ \ \phi \ge 0,
	\end{equation}
	where the last equality is deduced by \eqref{u0f0}.
	Hence \eqref{minore4bis} and \eqref{maggiore} give \eqref{def_distrp=1}.
	
	\bigskip

		{\it Identification of the field (i.e. proof of \eqref{def_zp=1}).} \\
	We first take $T_k(u_p)\varphi$ as a test function in the weak formulation of \eqref{pbp} obtaining
	\begin{equation*}	
	\int_{\Omega} |\nabla T_k(u_p)|^{p}\varphi + \int_{\Omega} T_k(u_p)|\nabla u_p|^{p-2}\nabla u_p \cdot \nabla \varphi + \int_{\Omega} g(u_p)T_k(u_p) \varphi= \int_{\Omega}  h(u_p)f T_k(u_p)\varphi,	
	\end{equation*}
	and then it follows from the Young inequality that
	\begin{equation}
	\begin{aligned}\label{z_1}
	\int_{\Omega} |\nabla T_k(u_p)|\varphi &+\int_{\Omega} T_k(u_p)|\nabla u_p|^{p-2}\nabla u_p \cdot \nabla \varphi + \int_{\Omega} g(u_p)T_k(u_p) \varphi 
	\\ 
	&\le \int_{\Omega}  h(u_p)f T_k(u_p) \varphi + \frac{p-1}{p}\int_{\Omega}\varphi.	
	\end{aligned}
	\end{equation} 
	Now we pass to the limit with respect to $k\to \infty$ in \eqref{z_1}; precisely we use the weak lower semicontinuity for the first term while for the second term we employ the strong convergence of $T_k(u_p)$ to $u_p$ in $W^{1,p}_0(\Omega)$. Moreover since $g(u_p)u_p\in L^1(\Omega)$ one can apply the Lebesgue Theorem in the remaining term on the left hand side, yielding to  
	\begin{equation}
	\begin{aligned}\label{z_2}
	\int_{\Omega} |\nabla u_p|\varphi &+\int_{\Omega} u_p|\nabla u_p|^{p-2}\nabla u_p \cdot \nabla \varphi + \int_{\Omega} g(u_p)u_p \varphi 
	\\ 
	&\le \int_{\Omega}  h(u_p)f u_p \varphi + \frac{p-1}{p}\int_{\Omega}\varphi.	
	\end{aligned}
	\end{equation}
 	Now we focus on the right hand side of \eqref{z_2}; if $\theta\ge 1$ we can simply pass to the limit as $p\to1$ by the Lebesgue Theorem. Otherwise we take $\delta>\overline{s}: \ \delta\not\in\{\eta: |\{u=\eta \}|>0\}$ and write
	$$\int_{\Omega}  h(u_p)f u_p \varphi = \int_{\{u_p\le \delta\}}  h(u_p)f u_p \varphi + \int_{\{u_p> \delta\}}  h(u_p)f u_p  \varphi,$$
	where we are allowed to pass to the limit as $p\to 1$ by the Lebesgue Theorem for the first term on the right hand side and by weak convergence for the second term. Indeed, since by Lemma \ref{lemmapriori0} $u_p^{1-\theta}$ is bounded in $L^{\frac{m}{m-1}}(\Omega)$, for the second term one has that $h(u_p)u_p\chi_{\{u_p> \delta\}}$ converges weakly to $h(u)u\chi_{\{u> \delta\}}$ in $L^{\frac{m}{m-1}}(\Omega)$.
	As concerns the left hand side of \eqref{z_2} we apply weak lower semicontinuity for the first term, the Fatou Lemma for the third term and finally the second term easily 
	passes to the limit. We get
	\begin{equation}
	\begin{aligned}\label{z_3}
	\int_{\Omega} \varphi|D u| +\int_{\Omega} uz \cdot \nabla \varphi + \int_{\Omega} g(u)u \varphi \le \int_{\Omega}  h(u)f u \varphi.	
	\end{aligned}
	\end{equation}	
	Now in order to manage the right hand side of \eqref{z_3} we prove that the following holds
	\begin{equation}\label{peru}	
		-u^* \operatorname{div}z +g(u)u = h(u)fu \ \text{ as measures in $\Omega$.}
	\end{equation}
	Indeed one can take in \eqref{dmloc} $(\rho_\epsilon\ast T_k(u))\varphi$ $(k>0)$ as a test function where $\rho_\epsilon$ is a sequence of standard mollifier, deducing
	\begin{equation*}\label{eqsigma}
	-\int_\Omega (\rho_\epsilon\ast T_k(u)) \varphi\, \operatorname{div}z + \int_\Omega g(u)(\rho_\epsilon\ast T_k(u))\varphi \ge  \int_\Omega h(u)f(\rho_\epsilon\ast T_k(u))\varphi,
	\end{equation*}
	and, as $\epsilon \to 0$, observing that $T_k(u)\in BV(\Omega)\cap L^\infty(\Omega)$ for the left hand side while applying the Fatou Lemma for the right hand side, one gets  
	\begin{equation}\label{eq1bis}
	-\int_\Omega T_k(u)^*\varphi\, \operatorname{div}z + \int_\Omega g(u)T_k(u)\varphi \ge  \int_\Omega h(u)fT_k(u)\varphi.
	\end{equation}	
	Moreover, since $g(u)u, h(u)fu \in L^1(\Omega)$ and $u$ satisfies \eqref{def_distrp=1}, one has from Lemma \ref{lempairing} that $u^*\in L^1_{\rm loc}(\Omega,\operatorname{div}z)$ and this allows to pass to the limit \eqref{eq1bis} with respect to $k$, taking to 
	\begin{equation}\label{eq1}
	-\int_\Omega u^*\varphi\, \operatorname{div}z + \int_\Omega g(u)u\varphi \ge \int_\Omega h(u)fu\varphi.
	\end{equation}	
	Now in order to show the reverse inequality we observe that, recalling $(z,Du)\le |Du|$ since $||z||_{L^\infty(\Omega)^N}\le 1$, \eqref{z_3} takes to 
	\begin{equation}
	\begin{aligned}\label{z_3bis}
	-\int_{\Omega}  u^* \varphi \operatorname{div}z + \int_{\Omega} g(u)u \varphi = \int_{\Omega} \varphi (z,Du) +\int_{\Omega} uz \cdot \nabla \varphi + \int_{\Omega} g(u)u \varphi \le \int_{\Omega}  h(u)f u \varphi,	
	\end{aligned}
	\end{equation}
	and then \eqref{eq1} and \eqref{z_3bis} imply \eqref{peru}. 
	Hence, using \eqref{peru} in \eqref{z_3}, one obtains
	\begin{equation*}
	\begin{aligned}\label{z_4}
	\int_{\Omega} \varphi|D u| +\int_{\Omega} uz \cdot \nabla \varphi \le -\int_\Omega u^*\varphi\, \operatorname{div}z,	
	\end{aligned}
	\end{equation*}
	that is
	\begin{equation}\label{primoverso}
	\int_{\Omega} \varphi|D u| \le \int_{\Omega}\varphi (z, D u),  \ \ \ \forall \varphi\in C^1_c(\Omega), \ \ \varphi \ge 0.	
	\end{equation}	
	Therefore \eqref{primoverso} gives \eqref{def_zp=1} since, as already observed, the reverse inequality is trivial.
	
	\medskip
	{\it Attainability of the boundary datum (i.e. proof of \eqref{def_bordop=1}).}	\\
	We take $T_k(u_p)$ as a test function in the weak formulation of \eqref{pbp} (recall that $u_p$ has zero trace on the boundary) 
	\begin{equation*}
	\int_{\Omega} |\nabla T_k(u_p)|^p + \int_{\Omega}g(u_p)T_k(u_p) + \int_{\partial \Omega}T_k(u_p) d\mathcal{H}^{N-1}\le  \int_{\Omega}  h(u_p)f T_k(u_p),
	\end{equation*}	
	and then it follows by the Young inequality that  
	\begin{equation*}\label{bordo1}
	\int_{\Omega} |\nabla T_k(u_p)| + \int_{\Omega}g(u_p)T_k(u_p) + \int_{\partial \Omega}T_k(u_p) d\mathcal{H}^{N-1} \le  \int_{\Omega}  h(u_p)f T_k(u_p) +\frac{p-1}{p}|\Omega|.
	\end{equation*}		
	We can take $p\to 1$ having
	\begin{equation*}
	\int_{\Omega} |D T_k(u)| + \int_{\partial \Omega}T_k(u) d\mathcal{H}^{N-1} \le  \int_{\Omega}  h(u)f T_k(u) - \int_{\Omega}g(u)T_k(u) = -\int_{\Omega}T_k(u)^*\operatorname{div}z,
	\end{equation*}	
	where the last equality follows from \eqref{eq1bis} and by taking $p\to 1$ in \eqref{z_1} (recall once again that $(z, DT_k(u))\le |DT_k(u)|$ and that $h(u)fT_k(u)\in L^1(\Omega)$).
    Now if one applies Lemma \ref{green} then 
    \begin{equation*}
	\int_{\Omega} |D T_k(u)| + \int_{\partial \Omega}T_k(u) d\mathcal{H}^{N-1} \le  \int_{\Omega}(z,D T_k(u)) - \int_{\partial \Omega} [T_k(u) z,\nu]d\mathcal{H}^{N-1},
	\end{equation*}	
	which gives the desired result since
    \begin{equation}\label{troncata}
	\int_{\Omega} |D T_k(u)| =  \int_{\Omega}(z,D T_k(u)).
	\end{equation}	
	Indeed one has 
	\begin{equation*}\label{theta}
	(z, DT_k(u)) = \lambda(z,DT_k(u),x) \, {|DT_k(u)|},
	\end{equation*}
	where \(\lambda(z, DT_k(u), \cdot)\) denotes the Radon-Nikodym derivative of $(z, DT_k(u))$ with respect to $|DT_k(u)|$. 
	Then it follows from Proposition $4.5$ of \cite{CDC} that
	\begin{equation*}
	\lambda(z,DT_k(u),x)=\lambda(z,Du,x),
	\qquad \text{for \(|DT_k(u)|\)-a.e.}\ x\in\Omega.
	\end{equation*}
	Moreover from \eqref{def_zp=1}
	\begin{equation*}\label{radonnik}
	\lambda(z,Du,x)=1,
	\qquad \text{for \(|Du|\)-a.e.}\ x\in\Omega,
	\end{equation*}
	and then we deduce 
	\begin{equation*}\label{radonnik2}
	\lambda(z,DT_k(u),x)=1,
	\qquad \text{for \(|DT_k(u)|\)-a.e.}\ x\in\Omega,
	\end{equation*}
	since $|DT_k(u)|$ is an absolutely continuous measure with respect to $|D u|$. This means that \eqref{troncata} holds. The proof is concluded.		
\end{proof}

\section{Uniqueness}
\label{sec:uni}

In this section we show that if a solution $u$ to \eqref{pb1} has finite energy and the absorption term $g$ is a non-decreasing function satisfying $g(u) \in L^1(\Omega)$ then the solution is unique provided that $h$ is decreasing (just non-increasing when $p>1$).
\subsection{The case $p>1$}
\begin{theorem}\label{unip>1}
	Let $1<p<N$, let $h$ be non-increasing, and let $g$ be non-decreasing then there is at most one solution $u_p\in W^{1,p}_0(\Omega)$ to problem \eqref{pb1} such that $g(u_p)\in L^1(\Omega)$. 
\end{theorem}
\begin{proof}
	The first part of the proof consists of an extension of the set of admissible test functions in \eqref{def_distr} from the set of $C^1_c(\Omega)$ to the one of $W^{1,p}_0(\Omega)\cap L^\infty(\Omega)$. Analogously to the proof of Lemma \ref{lemmal1} we consider a nonnegative $v\in W^{1,p}_0(\Omega)\cap L^\infty(\Omega)$ and also a sequence of nonnegative functions $v_{\eta,n}\in C^1_c(\Omega)$ having (we call $v_n$ the almost everywhere limit of $v_{\eta,n}$ as $\eta \to 0$)
	\begin{equation}\label{propapprox}
	\begin{cases}
		v_{\eta,n} \stackrel{\eta  \to 0}{\to} v_{n} \stackrel{n  \to \infty}{\to} v \ \ \ \text{in } W^{1,p}_0(\Omega) \text{ and } \text{$*$-weakly in } L^\infty(\Omega), \\
		\supp v_n\subset \subset \Omega: 0\le v_n\le v \ \ \ \text{for all } n\in \mathbb{N}.
	\end{cases}
	\end{equation}
	Recall that an example of such $v_{\eta,n}$ is given by $\rho_\eta \ast (v \wedge \phi_n)$ ($v \wedge \phi_n:= \inf (v,\phi_n)$) where $\rho_\eta$ is a smooth mollifier while $\phi_n$ is a sequence of nonnegative functions in $C^1_c(\Omega)$ which converges to $v$ in $W^{1,p}_0(\Omega)$.
	\\Hence if one takes $v_{\eta,n}$ as a test function in \eqref{def_distr}
	\begin{equation}\label{uni1}
	\int_{\Omega} |\nabla u_p|^{p-2}\nabla u_p\cdot\nabla 	v_{\eta,n}  + \int_\Omega g(u_p)	v_{\eta,n}= \int_{\Omega}h(u_p)f v_{\eta,n},
	\end{equation} 
	and we want to pass to the limit \eqref{uni1} as $\eta \to 0$. Indeed for the first term and the second term, recalling that $u_p\in W^{1,p}_0(\Omega)$ and $g(u_p)\in L^1(\Omega)$, we can pass to the limit since, by \eqref{propapprox}, $v_{\eta,n}$ converges to $v_n$ in $W^{1,p}_0(\Omega)$ and $*$-weakly in $L^\infty(\Omega)$. Finally for the term on the right hand side we observe that $h(u_p)f \in L^1_{\rm loc}(\Omega)$ and we can also pass here to the limit since $v_{\eta,n}$ converges $*$-weakly in $L^\infty(\Omega)$ to $v_n$ which has compact support in $\Omega$. Hence we deduce
	\begin{equation}\label{uni2}
	\int_{\Omega} |\nabla u_p|^{p-2}\nabla u_p\cdot\nabla v_n  + \int_\Omega g(u_p)v_n= \int_{\Omega}h(u_p)fv_n.
	\end{equation} 	
	Now an application of the Young inequality takes to 
	\begin{equation*}\label{uni3}
	\int_{\Omega}h(u_p)fv_n \le \int_{\Omega} |\nabla u_p|^{p} + \int_{\Omega}|\nabla v_n|^p  + ||g(u_p)||_{L^1(\Omega)}||v||_{L^\infty(\Omega)},	
	\end{equation*}	
	and it follows by \eqref{propapprox} that the right hand side of the previous is bounded with respect to $n$. Hence by the Fatou Lemma with respect to $n$, one gets
	\begin{equation*}\label{uni4}
	\int_{\Omega}h(u_p)fv \le C.	
	\end{equation*}	
	Now  we can easily pass to the limit as $n\to \infty$ in the first two terms of \eqref{uni2} as already done as $\eta \to 0$. For the right hand side of \eqref{uni2} we can apply the Lebesgue Theorem since 
	$$h(u_p)fv_n\le h(u_p)fv\in L^1(\Omega).$$
	Therefore it yields 
	\begin{equation}\label{uni5}
	\int_{\Omega} |\nabla u_p|^{p-2}\nabla u_p\cdot\nabla v  + \int_\Omega g(u_p)v = \int_{\Omega}h(u_p)fv,
	\end{equation} 		 
	for every $v\in W^{1,p}_0(\Omega)\cap L^\infty(\Omega)$. 
	\\Now we suppose the existence of two solutions $\overline{u}_p,\underline{u}_p$ to \eqref{pb1} such that $g(\overline{u}_p), g(\underline{u}_p) \in L^1(\Omega)$. As just proved $\overline{u}_p,\underline{u}_p$ satisfy \eqref{uni5}, which implies the allowance of taking $T_k(\overline{u}_p-\underline{u}_p)$ as a test function in difference between formulation \eqref{uni5} solved by $\overline{u}_p$ and the one solved by $\underline{u}_p$. Then it holds
	\begin{equation*}\label{uni6}
	\int_{\Omega} \left(|\nabla \overline{u}_p|^{p-2}\nabla \overline{u}_p - |\nabla \underline{u}_p|^{p-2}\nabla \underline{u}_p\right)\cdot\nabla T_k(\overline{u}_p-\underline{u}_p)   + \int_\Omega (g(\overline{u}_p)-g(\underline{u}_p))T_k(\overline{u}_p-\underline{u}_p) = \int_{\Omega}(h(\overline{u}_p)-h(\underline{u}_p))fT_k(\overline{u}_p-\underline{u}_p),
	\end{equation*} 	
	and it follows from the assumptions on $g,h$ the second term on the left hand side is nonnegative and the term on the right hand side is non-positive, giving
	\begin{equation*}\label{uni7}
	\int_{\Omega} \left(|\nabla \overline{u}_p|^{p-2}\nabla \overline{u}_p - |\nabla \underline{u}_p|^{p-2}\nabla \underline{u}_p\right)\cdot\nabla (\overline{u}_p-\underline{u}_p) \chi_{\{|\overline{u}_p-\underline{u}_p|\le k\}}  \le 0, \ \ \forall k>0,
	\end{equation*} 	
	that gives $\overline{u}_p=\underline{u}_p$ a.e. in $\Omega$ by a standard monotonicity argument. The proof is done.
\end{proof}
\subsection{The case $p=1$}
\begin{theorem}\label{uni_p=1}
Let $p=1$, let $f>0$ a.e. in $\Omega$, let $h$ be decreasing, and let $g$ be non-decreasing then there is at most one solution $u$ to problem \eqref{pb1} such that $g(u)\in L^1(\Omega)$. 
\end{theorem}
\begin{proof}

Let $u$ be a solution to \eqref{pb1}. If $h(0)=\infty$ then from $h(u)f\in L^1_{\rm loc}(\Omega)$ we deduce that up to a set of zero Lebesgue measure
$$\{u=0\}\subset \{f=0\},$$
which implies that $u>0$ a.e. in $\Omega$ since $f>0$ a.e. in $\Omega$. This means that it holds 
\begin{equation*}
-\operatorname{div}z + g(u) = h(u)f,
\end{equation*}
as measures in $\Omega$. Let us observe that, since $g(u)\in L^1(\Omega)$, one can apply Lemma \ref{lemmal1} from which one has that $\operatorname{div}z \in L^1(\Omega)$, namely $z\in \DM(\Omega)$. Hence a standard density argument implies that 
\begin{equation}\label{testestese}
-\int_{\Omega}v \operatorname{div}z + \int_{\Omega} g(u)v = \int_{\Omega} h(u)fv,
\end{equation}	
for all $v\in BV(\Omega)\cap L^\infty(\Omega)$.
Now suppose that $u_1$ and $u_2$ are solutions to \eqref{pb1} with fields, respectively, $z_1$ and $z_2$ and test with $T_k(u_1)- T_k(u_2)$ ($k>0$) the difference of weak formulations \eqref{testestese} solved by $u_1,u_2$,  and by also using Lemma \ref{green} one gets
\begin{equation*}
\begin{aligned}
&\int_\Omega (z_1, DT_k(u_1)) - \int_\Omega(z_1, DT_k(u_2)) -\int_\Omega(z_2, DT_k(u_1)) + \int_\Omega(z_2, DT_k(u_2))\\
&- \int_{\partial\Omega}(T_k(u_1)-T_k(u_2))[z_1,\nu])\, d\mathcal H^{N-1} +\int_{\partial\Omega}(T_k(u_1)-T_k(u_2))[z_2,\nu])\, d\mathcal H^{N-1} + \int_{\Omega} \left(g(u_1)-g(u_2)\right)(T_k(u_1)-T_k(u_2))						    	
\\
&= \int_\Omega (h(u_1)- h(u_2))f(T_k(u_1)-T_k(u_2)).
\end{aligned}
\end{equation*} 
Now reasoning as to deduce \eqref{troncata} one gets that $(z, DT_k(u_1))=|DT_k(u_1)|$ and that $(z, DT_k(u_2))=|DT_k(u_2)|$ as measures in $\Omega$. Furthermore we have (recalling also \eqref{def_bordop=1})
\begin{equation*}\label{unip=1_1}
\begin{aligned}
&\int_\Omega |DT_k(u_1)| - \int_\Omega(z_1, DT_k(u_2)) -\int_\Omega(z_2, DT_k(u_1)) + \int_\Omega |DT_k(u_2)|\\
&+  \int_{\partial\Omega}(T_k(u_1) +T_k(u_1) [z_2,\nu]  )\, d\mathcal H^{N-1} +\int_{\partial\Omega}(T_k(u_2) + T_k(u_2)[z_1,\nu] )\, d\mathcal H^{N-1}	
\\
&\le \int_\Omega (h(u_1)- h(u_2))f(T_k(u_1)-T_k(u_2))\le 0,
\end{aligned}
\end{equation*} 
and, since $||z_i||_{L^\infty(\Omega)^N} \le 1$ and $[z_i,\nu] \in [-1,1]$ for $i=1,2$, one gets 				    	
$$\displaystyle \int_\Omega (h(u_1)- h(u_2))f(T_k(u_1)-T_k(u_2))= 0,$$
which gives $T_k(u_1)=T_k(u_2)$ a.e. in $\Omega$ for every $k>0$. 
\end{proof}
\section{Some examples and generalizations}
\label{sec:example}
This section is devoted to examples and generalizations of problem \eqref{pb1}. At first we give two examples which, in some sense, suggest that Theorem \ref{teo_p>1} is optimal. Namely we present an example ($p=2$ and $h(s)=s^{-\theta}, (0<\theta<1)$) in which we find an infinite energy solution to \eqref{pb1} if $q<\frac{1- m\theta}{m-1}$. The second example shows that there exist unbounded solutions to \eqref{pb1} when $p=1$ regardless of the choice of $q$. Concerning the generalizations we provide a particular case where the absorption term $g$ gives always rise to the existence of bounded solutions to \eqref{pb1}. Moreover we also deal with more general operators as well as infinite energy solutions when $p>1$. 

\subsection{Examples}

We start with an example which shows that in case $p=2$ and $q<\frac{1- m\theta}{m-1}$ there exists a solution to \eqref{pb1} $u\not\in H^1_0(\Omega)$.  
\begin{example}\label{examplesharp}
	We denote by $r^*=\frac{rN}{N-r}$, and by $r^{**}=(r^*)^*$.
	It is well known that for any $0<\theta<1$ and $\frac{N(\theta+1)}{N+2\theta}<r<\frac{2N}{N+2}$ there exists a nonnegative $f\in L^r(\Omega)$ such that
	\begin{equation*}
	\begin{cases}
	\dis -\Delta u = f & \text{in}\;\Omega,\\
	u=0 & \text{on}\;\partial \Omega,
	\end{cases}
	\end{equation*}
    admits a nonnegative solution $u\not\in H^1_0(\Omega)$ which only belongs to $W_0^{1,r}(\Omega)\cap L^{r^{**}}(\Omega)$. Hence $u$ also solves
	\begin{equation*}
	\begin{cases}
	\dis -\Delta u + u^q = \frac{(f + u^q)u^\theta}{u^\theta} & \text{in}\;\Omega,\\
	u=0 & \text{on}\;\partial \Omega,
	\end{cases}
	\end{equation*}	
	where $q\ge 0$. We note that if $q\le \frac{r^{**}}{r}$ then $(f + u^q)u^\theta\in L^m(\Omega)$ with
	$$m=\frac{r^{**}r}{r\theta + r^{**}},$$
	which, under the assumption on $r$, gives $m>1$. Hence, fixing $q= \frac{r^{**}}{r}$, one gets 
	$$q= \frac{r^{**}}{r} < \frac{1-m\theta}{m-1}= \frac{r^{**}(1-r\theta) +r\theta}{r^{**}(r-1) - r\theta}.$$
	We also stress that $q$ and $\frac{1-m\theta}{m-1} \to 2^*-1$ as $r\to \frac{2N}{N+2}$ and $u$, in this case, belongs to $H^1_0(\Omega)$. Hence, we have shown that there always exist data $f\in L^m(\Omega)$ such that one can find a solution $u$ with infinite energy for any $q<\frac{1-m\theta}{m-1}$.
\end{example}
Now we give an example in the case of the $1$-Laplace operator which shows that the presence of the absorption term may not, in general, guarantee the boundedness of solutions to \eqref{pb1}. We recall that for a nonnegative $f\in L^{N,\infty}(\Omega)$ the problem 
\begin{equation*}
\begin{cases}
\dis -\Delta_1 u = \frac{f}{u^\theta} & \text{in}\;\Omega,\\
u=0 & \text{on}\;\partial \Omega,
\end{cases}
\end{equation*}	
has a bounded solution when $\theta\le 1$ as proved in \cite{dgop} and \cite{dgs}. 
Let us consider 
\begin{equation}
\label{example}
\begin{cases}
\dis -\Delta_1 u + u^q = \frac{f}{u^\theta} & \text{in}\;B_R(0),\\
u=0 & \text{on}\;\partial B_R(0),
\end{cases}
\end{equation}		
where $q,\theta \ge 0$, $R>0$ and we set $r=|x|$.\\
We show that the regularizing effect given by $q$ is not sufficient to provide bounded solutions to the problem when the datum does not belong to $L^{N,\infty}(\Omega)$. Indeed, even if $q$ is large one can always find a datum $f$ which is \textit{almost} in $L^{N,\infty}(\Omega)$ and for which the problem admits an unbounded solution.
\begin{example}\label{exampleunbounded}
Let us consider problem \eqref{example} where $\displaystyle f= Nr^{-1-\theta\alpha}$ ($\alpha>0$). We look for an unbounded radial solution $u(r)$ such that $u'(r)<0$. In this case the vector field $z$ is given by $z(x)=-\frac{x}{r}$ and $-\operatorname{div}z= \frac{N-1}{r}$. This means that the boundary condition is satisfied as $[z,\nu]=-1$  and then a solution of the form $u(r)= r^{-\alpha}$ needs to satisfy
$$\frac{N-1}{r} + \frac{1}{r^{\alpha q}} = \frac{N}{r}.$$
Hence $u(r)= r^{-\frac{1}{q}}$ solves \eqref{example}. We explicitly observe that, formally, if one takes $q\to\infty$ then $f$ turns out to be in $L^{N,\infty}(\Omega)$ and $u$, as expected, is bounded.
\end{example}

\subsection{Infinite energy solutions}
The aim of this section is twofold: first of all we are interested in considering problems with more general operators and absorption terms. Secondly, we deal with data which take us out of the finite energy setting but, as we will see below, the absorption term will provide a regularizing effect even in this case. Let $1<p<N$ and consider the following problem
\begin{equation}
	\label{pbgen}
	\begin{cases}
		\dis -\operatorname{div}(a(x,\nabla u)) + g(x,u) = h(u)f & \text{in}\;\Omega,\\
		u=0 & \text{on}\;\partial\Omega.
	\end{cases}
\end{equation}
where $\displaystyle{a(x,\xi):\Omega\times\mathbb{R}^{N} \to \mathbb{R}^{N}}$ is a Carath\'eodory function satisfying the classical Leray-Lions structure conditions, namely
\begin{align}
&a(x,\xi)\cdot\xi\ge \alpha|\xi|^{p}, \ \ \ \alpha>0,
\label{cara1}\\
&|a(x,\xi)|\le \beta|\xi|^{p-1}, \ \ \ \beta>0,
\label{cara2}\\
&(a(x,\xi) - a(x,\xi^{'} )) \cdot (\xi -\xi^{'}) > 0,
\label{cara3}	
\end{align}
for every $\xi\neq\xi^{'}$ in $\mathbb{R}^N$ and for almost every $x$ in $\Omega$.
\\Once again $f\in L^m(\Omega)$ with $m>1$ and $h:[0,\infty)\to [0,\infty]$ is a continuous function, possibly singular with $h(0)\not=0$, and finite outside the origin satisfying \eqref{h1} with $\gamma<1$ and \eqref{h2} with $\theta<1$.
The absorption term $g:\Omega \times [0,\infty)\to [0,\infty)$ is continuous and $g(x,0)=0$ for almost every $x \in \Omega$. Moreover we also require that
\begin{equation}\label{gloc}
	\sup_{s\in[0,t]} g(x,s) \in L^1_{\rm loc}(\Omega), \ \ \forall t\ge 0,
\end{equation}
and the following growth condition at infinity
\begin{equation}\label{g2}\tag{g2}
	\exists q: \ \frac{p-1-pm\theta}{pm-p+1} < q < \frac{1-m\theta}{m-1}, \  \exists\nu, s_1>0 : g(x,s)\geq \nu s^{q} \ \text{for all}\;s\geq s_1.
\end{equation}
First of all we remark that, under the assumptions listed above, Theorem \ref{teo_p>1} still holds when $g$ satisfies \eqref{g1} in place of \eqref{g2} with minor modifications in the proof. In this section, as already remarked, we are interested in extending the above cited theorem when the regularizing effect given by $g$ is not sufficient in order to expect $W^{1,p}_0$-solutions. Namely it holds the following result of which we only give the idea of the proof.
\begin{theorem}\label{teo_p>1gen}
	Let $1<p<N$, $0\le f\in L^{m}(\Omega)$ with $m>1$, let $h$ satisfy \eqref{h1} with $\gamma< 1$ and \eqref{h2} with $\theta<1$, and let $g$ satisfy \eqref{gloc} and \eqref{g2}.
	Then there exists a solution to problem \eqref{pb1} which belongs to $W^{1,r}_0(\Omega)$ with $r=\frac{p(q+\theta)m}{q+1}$.
\end{theorem}  
\begin{proof}
	We only sketch the proof which relies (also in this case) on an approximation argument. We consider a nonnegative $u_n\in W^{1,p}_0(\Omega)\cap L^\infty(\Omega)$ solution to
\begin{equation}\label{pbngen}
	\begin{cases}
		\dis -\operatorname{div} (a(x,\nabla u_{n})) + g(x,u_{n})= h_n(u_{n})f_n & \text{in}\;\Omega,\\
		u_{n}=0 & \text{on}\;\partial\Omega.
	\end{cases}
\end{equation}
whose existence can be proved as for the solutions to \eqref{pbpn2}.\\
Let us denote by $\eta=(q+\theta)m -q -\gamma$ and observe that assumptions on $q$ imply $0<\eta+\gamma<1$. Hence we test the weak formulation of \eqref{pbngen} with $(u_n + \varepsilon)^{\eta+\gamma} - \varepsilon^{\eta+\gamma}$ and we have 
\begin{equation}\label{gen0}
	(\eta+\gamma)\int_\Omega \frac{|\nabla u_n|^p}{(u_n+\varepsilon)^{1-\eta-\gamma}} +\int_{\Omega} g(x,u_n) ((u_n + \varepsilon)^{\eta+\gamma} - \varepsilon^{\eta+\gamma})\le \int_{\Omega} h_n(u_{n})f_n (u_n + \varepsilon)^{\eta+\gamma} , 
\end{equation}
which, getting rid of the first term, taking $\epsilon$ to zero, and using \eqref{g2}, takes to (suppose $s_1\le\overline{s}$)
\begin{equation}\label{gen1}
\nu\int_{\{u_n> s_1\}} u_n^{q+\eta+\gamma} \le \underline{s}^{\eta}\int_{\{u_n< \underline{s}\}} f  +\max_{s\in [\underline{s}, \overline{s}]} [h(s)s^{\eta+\gamma}] \int_{\{\underline{s}\le u_n\le \overline{s}\}} f + C_{\delta}\int_{\Omega} f^m + \delta\int_{\{u_n> s_1\}} u_n^{\frac{(\eta+\gamma-\theta)m}{m-1}}. 
\end{equation}
Our choice of $\eta$ implies $\frac{(\eta+\gamma-\theta)m}{m-1}=q+\eta+\gamma$ and then \eqref{gen1} gives that $u_n$ is bounded in $L^{(q+\theta)m}(\Omega)$ with respect to $n$. Getting rid of the second term in \eqref{gen0} and using that $u_n$ is bounded in $L^{(q+\theta)m}(\Omega)$ allows also to deduce that
\begin{equation*}
	\int_\Omega \frac{|\nabla u_n|^p}{(u_n+\varepsilon)^{1-\eta-\gamma}} \le C, 
\end{equation*}
for a constant $C$ independent from $n$. 
\\Now from the Young inequality one has that
\begin{equation*}
\begin{aligned}
\int_\Omega |\nabla u_n|^r 
&= \int_\Omega \frac{|\nabla u_n|^r}{(u_n+\varepsilon)^{\frac{(1-\eta-\gamma)r}{p}}}(u_n+\varepsilon)^\frac{(1-\eta-\gamma)r}{p} \le \int_\Omega \frac{|\nabla u_n|^p}{(u_n+\varepsilon)^{1-\eta-\gamma}} + \int_\Omega (u_n+\varepsilon)^\frac{(1-\eta-\gamma)r}{p-r},\\
&\le C + \int_\Omega (u_n+\varepsilon)^\frac{(1-\eta-\gamma)r}{p-r} = C + \int_\Omega (u_n+\varepsilon)^{(q+\theta)m}\le C,
\end{aligned}
\end{equation*}
where the last equality follows from the choice of $r$. Once that the previous estimate holds then the existence of a solution with arguments similar to the ones of the proof of Theorem \ref{teo_p>1}. 
\end{proof}

\subsection{A particular case of bounded solutions}

In the spirit of \cite{arbo,arbo2} we consider a particular case of \eqref{pbgen}, namely 

\begin{equation}
	\label{pbgen2}
	\begin{cases}
		\dis -\operatorname{div}(a(x,\nabla u)) + V u = h(u)f & \text{in}\;\Omega,\\
		u=0 & \text{on}\;\partial\Omega.
	\end{cases}
\end{equation}
where $\displaystyle{a(x,\xi):\Omega\times\mathbb{R}^{N} \to \mathbb{R}^{N}}$ once again satisfies \eqref{cara1}, \eqref{cara2} and \eqref{cara3}.
Here $V,f$ are nonnegative functions in $L^1(\Omega)$ such that 
\begin{equation}\label{V}
	f(x)\le V(x) \ \ \text{for almost every }x\in \Omega.
\end{equation}
As before $h:[0,\infty)\to [0,\infty]$ is continuous and possibly singular in zero (with $h(0)\not=0$) which satisfies \eqref{h1} and which it is bounded at infinity, namely it satisfies \eqref{h2} with $\theta=0$.
\\Under the above set of hypotheses we prove that the regularizing effect given by $V$ implies the existence of a bounded (and with finite energy) solution to problem \eqref{pb1}; we remark that this has been already proven in \cite{arbo} in case $p>1$ and $h(s)=1$. 
\begin{theorem}\label{teo_p>1gen2}
	Let $1\le p<N$, let $0\le f\in L^{1}(\Omega)$ satisfy \eqref{V}, let $h$ satisfy \eqref{h1} with $\gamma\le1$ and \eqref{h2} with $\theta=0$.
	Then there exists a bounded solution $u$ to problem \eqref{pb1}. Moreover if $1<p<N$ then $u$ belongs to $W^{1,p}_0(\Omega)$.
\end{theorem}  
\begin{proof}
	As for Theorem \ref{teo_p>1gen} we just sketch the proof of the main estimates. Let us consider a nonnegative $u_n\in W^{1,p}_0(\Omega)\cap L^\infty(\Omega)$, which is a solution to
	\begin{equation}\label{pbngen2}
		\begin{cases}
			\dis -\operatorname{div} (a(x,\nabla u_{n})) + V u_n = h_n(u_{n})f_n & \text{in}\;\Omega,\\
			u_{n}=0 & \text{on}\;\partial\Omega.
		\end{cases}
	\end{equation}
	whose existence, once again, follows as for the solutions to \eqref{pbpn2}.\\
	Let us take $G_{k}(u_n)$ ($k>0$) as a test function in the weak formulation of \eqref{pbngen2} then
	\begin{equation*}
		\alpha\int_\Omega |\nabla G_{k}(u_n)|^p + \int_\Omega Vu_nG_{k}(u_n) \le \int_\Omega h(u_n)fG_{k}(u_n) \le \sup_{s\in[k,\infty)}[h(s)]\int_\Omega V G_{k}(u_n),
	\end{equation*}
	which gives
	\begin{equation*}
	\alpha\int_\Omega |\nabla G_{k}(u_n)|^p + \int_\Omega V[u_n-\sup_{s\in[k,\infty)}[h(s)]]G_{k}(u_n) \le 0.
	\end{equation*}	
	Hence one can always fix $k=\overline{k}$ (for some $\overline{k}>0$ independent of $n$ and $p$) sufficiently large such that the second term on the left hand side of the previous is nonnegative. This implies that $||u_n||_{L^\infty(\Omega)}\le \overline{k}$. Now if one takes $u_n$ itself as a test function in the weak formulation of \eqref{pbngen2} then it follows that $||u_n||_{W_0^{1,p}(\Omega)}\le C$ with $C$ independent of $n$ and $p$. The above estimates allows to reason as in the proof of Theorem \ref{teo_p>1} in order to conclude. 
\end{proof}
\section{Locally finite energy solutions in presence of strong singularities}
\label{sec:strong}

Up to now we have focused on various features of problem \eqref{pb1} when $h$ satisfies \eqref{h1} with $\gamma\le1$. The aim of this section is tackling the case of a function $h$ blowing up faster at the origin. Hence here we refer to \eqref{pb1} with $h$ as a continuous function (possibly unbounded at the origin) satisfying \eqref{h1} with $\gamma>1$ and \eqref{h2}. Moreover the function $g$ is continuous, $g(0)=0$ and, analogously to \eqref{g1}, if $\theta<1$ we require the following growth condition at infinity  
\begin{equation}\label{g3}\tag{g3}
\exists q: \ q \ge \frac{\gamma-m\theta}{m-1}, \  \exists\nu, s_1>0 : g(s)\geq \nu s^{q} \ \text{for all}\;s\geq s_1.
\end{equation}

In case $\gamma>1$ problem \eqref{pb1} is quite different; for instance let us think to the case $p>1$, $g\equiv 0$ and $f$ belonging to a suitable Lebesgue space, then we only have global estimates on some power of the solution $u$ in $W^{1,p}_0(\Omega)$ which can be formally deduced by taking $u^\gamma$ as test function in the weak formulation of \eqref{pb1}. This (and other estimates) allows us to deduce local estimates on the solution itself. Otherwise if $f$ is not sufficiently regular we are just able to prove that, in general, every truncation of the solution has locally finite energy. For this kind of effects and even more we refer to \cite{bo,do,op,op2}. When $p=1$ the same effect arises: for instance in \cite{dgop} it is proved the existence of a locally $BV$-solution if $f$ is in $L^N(\Omega)$. Here we prove that if $f\in L^m(\Omega)$ with $m\ge 1$ ($m>1$ if $\theta<1$) then $q$ can be chosen sufficiently large such that there exists a solution to \eqref{pb1} with local finite energy. 
 The discussion above takes naturally to a suitable {\it localization} of the notions of solution given by Definitions \ref{weakdefp>1} and \ref{weakdefpositive}. 
\begin{defin}\label{distributional}
    Let $1<p<N$ then $u\in W^{1,1}_{\rm loc}(\Omega)$ such that $|\nabla u|^{p-1} \in L^1_{\rm loc}(\Omega)$ is a solution to problem \eqref{pb1} if $g(u), h(u)f \in L^1_{\rm loc}(\Omega)$, it holds
	\begin{equation}
	G_{k}(u)\in W^{1,p}_{0}(\Omega) \ \ \ \text{for all}\ k>0\,,\label{pbordodef}
	\end{equation}
	and
	\begin{equation}\int_{\Omega}|\nabla u|^{p-2} \nabla u\cdot \nabla \varphi + \int_{\Omega} g(u)\varphi = \int_{\Omega}h(u)f\varphi, \ \ \ \forall \varphi\in C^1_c(\Omega). \label{pweakdef}
	\end{equation}
\end{defin}
\begin{remark}
	We remark that condition \eqref{pbordodef} is the way the boundary datum is achieved. When $g\equiv 0$ this kind of request is already present in \cite{cst,dgop} and, in particular, in \cite{cst}, the authors prove uniqueness of solutions in $W^{1,p}_{\rm loc}(\Omega)$ when $h(s)=s^{-\gamma}$ for suitable data and a regular domain. A different request for the boundary condition, in case $\gamma>1$, is that 
	\begin{equation*}
	u^{\frac{\gamma-1+p}{p}}\in W^{1,p}_{0}(\Omega),
	\end{equation*}
	as one can find in \cite{bo,do,op,op2}.
\end{remark}
Then we give the one for $p=1$.
\begin{defin}
	\label{weakdefpositivestrong}
	Let $p=1$ then a nonnegative $u\in BV_{\rm loc}(\Omega)$ such that $\chi_{\{u>0\}} \in BV_{\rm loc}(\Omega)$ and $u^\gamma \in BV(\Omega)$ is a solution to problem \eqref{pb1} if $g(u), h(u)f \in L^1_{\rm loc}(\Omega)$ and if there exists $z\in \mathcal{D}\mathcal{M}^\infty_{\rm loc}(\Omega)$ with $||z||_{L^\infty(\Omega)^N}\le 1$ such that
	\begin{align}
		&-\psi^*(u)\operatorname{div}z + g(u) =  h(u)f \ \ \text{as measures in } \Omega, \label{def_distrp=1strong}
		\\
		&\text{where } \psi(s) = \begin{cases}1 \ \ \ &\text{if } h(0)<\infty, \\ \chi_{\{s>0\}} \ \ \ &\text{if } h(0)=\infty, \end{cases} \nonumber
		\\
		&(z,Du)=|Du| \label{def_zp=1strong} \ \ \ \ \text{as measures in } \Omega,
		\\
		&T_k(u^\gamma(x)) + [T_k(u^{\gamma})z,\nu] (x)=0 \label{def_bordop=1strong}\ \ \ \text{for  $\mathcal{H}^{N-1}$-a.e. } x \in \partial\Omega \ \text{and for every} \ k>0.
	\end{align}
\end{defin} 
Hence we state and sketch the proof of the following existence theorem.
\begin{theorem}\label{teo_p>1strong}
	Let $1\le p<N$, let $h$ satisfy \eqref{h1} with $\gamma>1$, \eqref{h2} and suppose that one of the following assumptions hold:
	\begin{itemize}
		\item[i)] $\theta\ge \gamma \text{ and } f\in L^1(\Omega)$;
		
		\item[ii)] $\theta<\gamma,f\in L^{m}(\Omega)$ with $m>1$, and $g$ satisfies \eqref{g3}. 
	\end{itemize}	
	Then there exists a solution $u$ in the sense of Definition \ref{distributional} and \ref{weakdefpositivestrong} to problem \eqref{pb1}. Moreover if $1<p<N$ then $u$ belongs to $W^{1,p}_{\rm loc}(\Omega)$ and $g(u)u^\gamma\in L^1(\Omega)$.
\end{theorem}
\begin{proof}
	We provide the main hints of the proof, namely how to gain the estimates for the approximate solutions $u_n$ to problem \eqref{pbpn2}. Moreover, as already done in proving Theorem \ref{teo_p>1} we will show that these estimates are independent from both $n$ and $p$. This will allow us to deduce the existence of a solution when $p>1$ and later, by moving $p$, we will get the result for $p=1$.\\ 
	Hence if one takes $u_n^\gamma$ as a test function in the weak formulation of \eqref{pbpn2} then 
	\begin{equation*}\label{stimastrong1}
	\begin{aligned}
		&\left(\frac{p}{\gamma-1+p}\right)^p\gamma\int_\Omega |\nabla u_n^{\frac{\gamma-1+p}{p}}|^p + \int_\Omega g(u_n)u_n^\gamma\le \underline{C}\int_{\{u_n< \underline{s}\}} f + \max_{s\in[\underline{s},\overline{s}]}[h(s)s^\gamma] \int_{\{\underline{s}\le u_n \le \overline{s}\}} f + \overline{C}\int_{\{u_n> \overline{s}\}} f u_n^{\gamma-\theta}.
	\end{aligned}
	\end{equation*}
	In case i) one has that the right hand side of the previous is bounded and the estimate is done. Otherwise in case ii) one applies the Young inequality obtaining
	\begin{equation}\label{stimastrong1bis}
	\begin{aligned}
	\left(\frac{p}{\gamma-1+p}\right)^p\gamma\int_\Omega |\nabla u_n^{\frac{\gamma-1+p}{p}}|^p + &\int_\Omega g(u_n)u_n^\gamma\le \underline{C}\int_{\{u_n< \underline{s}\}} f + \max_{s\in[\underline{s},\overline{s}]}[h(s)s^\gamma] \int_{\{\underline{s}\le u_n \le \overline{s}\}} f + \overline{C}\int_{\{u_n> \overline{s}\}} f u_n^{\gamma-\theta}
	\\
	&\le \left(\underline{C}+\max_{s\in[\underline{s},\overline{s}]}[h(s)s^\gamma]\right) ||f||_{L^1(\Omega)} + C_\varepsilon\int_{\{u_n> \overline{s}\}} f^m + \varepsilon\int_{\{u_n> \overline{s}\}} u_n^\frac{(\gamma-\theta)m}{m-1}.
	\end{aligned}
	\end{equation}		
	Now recalling \eqref{g3} and applying the Young inequality if $q > \frac{\gamma-m\theta}{m-1}$ (if $q = \frac{\gamma-m\theta}{m-1}$ is not necessary) one has (without loss of generality assume that $s_1\le  \overline{s}$)
	\begin{equation*}
\begin{aligned}
&\left(\frac{p}{\gamma-1+p}\right)^p\gamma\int_\Omega |\nabla u_n^{\frac{\gamma-1+p}{p}}|^p + \nu\int_{\{u_n>s_1\}} u_n^{q+ \gamma}
\\
&\le \left(\underline{C}+\max_{s\in[\underline{s},\overline{s}]}[h(s)s^\gamma]\right) ||f||_{L^1(\Omega)} + C_\varepsilon||f||^m_{L^m(\Omega)} + \varepsilon \int_{\{u_n> s_1\}}u_n^{q+\gamma} + C|\Omega|,
\end{aligned}
\end{equation*}	
 and, fixing $\varepsilon$ small enough, one has that $u_n$ is bounded in $L^{q+\gamma}(\Omega)$ with respect to $n$ and $p$. Using this information in \eqref{stimastrong1bis} one has that in both cases i) and ii) the following holds
 	\begin{equation}\label{stimastrong2}
 \begin{aligned}
 &\int_\Omega |\nabla u_n^{\frac{\gamma-1+p}{p}}|^p + \int_\Omega g(u_n)u_n^\gamma\le C,
 \end{aligned}
 \end{equation}
 for some constant $C$ not dependent on $n$ and $p$.
 Now we take $G_k(u_n)$ ($k>0$) as a test function in the weak formulation of \eqref{pbpn2} and if one gets rid of the absorption term then
  \begin{equation*}
 \begin{aligned}
\int_\Omega |\nabla G_k(u_n)|^p \le \int_{\{u_n\le\underline{s}\}}h(u_n)fG_k(u_n) + \overline{C}\int_{\{u_n>\overline{s}\}}fG_k(u_n)^{1-\theta}.
 \end{aligned}
 \end{equation*}
 Now if $\theta\ge 1$ it is simple to show that the right hand side is bounded by a positive constant independent of $n$ and $p$. Otherwise if $\theta<1$ it follows from the Young inequality that  
 \begin{equation*}
 \begin{aligned}
 \int_\Omega |\nabla G_k(u_n)|^p &\le \int_{\{u_n\le\underline{s}\}}h(u_n)fG_k(u_n) + \overline{C}\int_{\{u_n>\overline{s}\}}fG_k(u_n)^{1-\theta}
 \\
 &\le \max_{s\in[k, \overline{s}]}[h(s)s] ||f||_{L^1(\Omega)} + \overline{C}\int_{\Omega}f^m + \overline{C}\int_\Omega u_n^\frac{(1-\theta)m}{m-1}\le C 
 \end{aligned}
 \end{equation*}
 since $u_n$ is bounded in $L^{q+\gamma}(\Omega)$. In all cases one has that $||G_k(u_n)||^p_{W^{1,p}_0(\Omega)} \le C$ where $C$ does not depend on $n$ and $p$.
 \\ \\Now let us consider $\varphi\in C^1_c(\Omega)$ such that $0\le \varphi \le 1$ and let us take the non-positive $(T_k(u_n)-k)\varphi^p$ as a test function in the weak formulation of \eqref{pbpn2}. Hence one has 
 \begin{equation*}
 	\int_{\Omega} |\nabla T_k(u_n)|^{p}\varphi^p + 	p\int_{\Omega} |\nabla u_n|^{p-2}\nabla u_n\cdot \nabla \varphi (T_k(u_n)-k)\varphi^{p-1} + \int_{\Omega} g(u_n)(T_k(u_n)-k)\varphi^p = \int_\Omega h_n(u_n)f_n(T_k(u_n)-k)\varphi^p \le 0,
 \end{equation*}
 which implies that
  \begin{equation*}\label{stimalocgamma>1}
 \int_{\Omega} |\nabla T_k(u_n)|^{p}\varphi^p \le 	pk\int_{\Omega} |\nabla T_k(u_n)|^{p-1}|\nabla \varphi| \varphi^{p-1} +pk\int_{\Omega} |\nabla G_k(u_n)|^{p-1}|\nabla \varphi| \varphi^{p-1} + k\int_{\Omega} g(u_n).
 \end{equation*}
 Finally, observing that \eqref{stimastrong2} implies that $g(u_n)$ is bounded in $L^1(\Omega)$ and recalling that $G_k(u_n)$ is bounded in $W^{1,p}_0(\Omega)$, it follows from the Young inequality that
   \begin{equation*}
 \int_{\Omega} |\nabla T_k(u_n)|^{p}\varphi^p \le  	pk\varepsilon \int_{\Omega} |\nabla T_k(u_n)|^{p}\varphi^{p} + pk C_\varepsilon\int_{\Omega} |\nabla \varphi|^p + C,
 \end{equation*}
 where $C$ does not depend on $n$ and $p$. Therefore, fixing $\varepsilon$ small enough one has that
 \begin{equation*}
 	||T_k(u_n)||_{W^{1,p}(\omega)}\le C, \ \ \ \forall \omega\subset\subset \Omega,
 \end{equation*}
 for a positive constant $C$ independent on $n$ and which is bounded as $p\to 1$. This gives that $u_n$ is locally bounded in $W^{1,p}(\Omega)$ and then, repeating the arguments of Theorem \ref{teo_p>1} with the adjustment of getting test functions in $W^{1,p}_0(\Omega)\cap L^\infty(\Omega)$ having compact support, one is able to prove that there exists $u_p\in W^{1,p}_{\rm loc}(\Omega)$ such that
 \begin{equation}\label{testgamma>1}
 	\int_{\Omega}|\nabla u_p|^{p-2}\nabla u_p\cdot \nabla \varphi + \int_\Omega g(u_p)\varphi = \int_\Omega h(u_p)f\varphi, 
 \end{equation}
 for all $\varphi \in W^{1,p}_0(\Omega)\cap L^\infty(\Omega)$ having compact support in $\Omega$.\\
 Now that we have a solution $u_p$ to problem \eqref{pb1} in case $p>1$ we need to move $p\to 1$ as already done in the proof of Theorem \ref{teo_p>1}, which we will retrace highlighting the main differences. Since $u_p$ is locally bounded in $W^{1,p}(\Omega)$ and reasoning as in \eqref{bv} one gets that $u_p$ is locally bounded in $BV(\Omega)$. Moreover an analogous standard diagonal arguments takes to a vector field $z$ with $||z||_{L^\infty(\Omega)^N}\le 1$. Moreover from \eqref{stimastrong2} one also has that $u_p^{\frac{\gamma-1+p}{p}}$ is bounded in $BV(\Omega)$ with respect to $p$. Hence there exists $u\in BV_{\rm loc}(\Omega)$ such that $u^\gamma \in BV(\Omega)$ with $u_p^{\frac{\gamma-1+p}{p}}$ converging to $u^\gamma$ in $L^r(\Omega)$ with $r<\frac{N}{N-1}$ and $\nabla u_p^{\frac{\gamma-1+p}{p}}$ converging $*$-weakly as measures to $Du^\gamma$. The distributional formulation \eqref{def_distrp=1strong} and the fact that $z\in \DM_{\rm loc}(\Omega)$ can be proved as in proof of Theorem \ref{teo_p>1}. For what concerns the \eqref{def_zp=1strong} one can reason similarly to what done in \cite{dgop}. Here we highlight a sketch of the proof.  We take $T_k^\frac{\gamma-1+p}{p}(u_p)\varphi$ with $0\le \varphi\in C^1_c(\Omega)$ as a test function in the weak formulation of \eqref{testgamma>1} and we apply the Young inequality, obtaining
 \begin{equation}\label{stimak>1}	
 \begin{aligned}
 &\left(\frac{\gamma-1+p}{p}\right)^\frac{1}{p}\left(\frac{p^2}{\gamma-1+p^2}\right)\int_{\Omega} |\nabla T_k^{\frac{\gamma-1+p^2}{p^2}}(u_p)|\varphi + \int_{\Omega} T_k^\frac{\gamma-1+p}{p}(u_p)|\nabla u_p|^{p-2}\nabla u_p \cdot \nabla \varphi 
 \\
 &+\int_{\Omega} g(u_p) T_k^\frac{\gamma-1+p}{p}(u_p)\varphi\le \int_{\Omega}  h(u_p)f u_p^\frac{\gamma-1+p}{p} \varphi  + \frac{p-1}{p}\int_{\Omega}\varphi.	
 \end{aligned}
 \end{equation}
 We observe that the second term on the left hand side is bounded with respect to $k$ and the third term is nonnegative. The right hand side is finite, indeed for the first term one has 
 $$\int_{\Omega}  h(u_p)f u_p^\frac{\gamma-1+p}{p} \varphi \le \underline{s}^\frac{\gamma-1+p}{p} \int_{\{u_p<\underline{s}\}}  h(u_p)f\varphi + \max_{s\in [\underline{s}, \overline{s}]} [h(s)s^\frac{\gamma-1+p}{p}]\int_{\{\underline{s} \le u_p \le \overline{s}\}}  f  \varphi + \overline{C}\int_{\{u_p>\overline{s}\}}  f u_p^\frac{\gamma-1+p - p\theta}{p} \varphi,$$
 and, since by \eqref{stimastrong2} one has that $u_p$ is bounded in $L^{q+\gamma}(\Omega)$, then  $u_p^\frac{\gamma-1+p - p\theta}{p}$ belongs to $L^{\frac{m}{m-1}}(\Omega)$ which gives that $ h(u_p)f u_p^\frac{\gamma-1+p}{p} \varphi\in L^1(\Omega)$ (recall that $h(u_p)f\in L^1_{\rm loc}(\Omega)$). Hence one can simply take $k\to \infty$ in \eqref{stimak>1}, applying the Fatou Lemma for the third term on the left hand side, yielding to
 \begin{equation}\label{stimak>2}	
 \begin{aligned}
 &\left(\frac{\gamma-1+p}{p}\right)^\frac{1}{p}\left(\frac{p^2}{\gamma-1+p^2}\right)\int_{\Omega} |\nabla u_p^{\frac{\gamma-1+p^2}{p^2}}|\varphi + \int_{\Omega} u_p^\frac{\gamma-1+p}{p}|\nabla u_p|^{p-2}\nabla u_p \cdot \nabla \varphi 
 \\
 &+\int_{\Omega} g(u_p) u_p^\frac{\gamma-1+p}{p}\varphi \le \int_{\Omega}  h(u_p)f u_p^\frac{\gamma-1+p}{p} \varphi  + \frac{p-1}{p}\int_{\Omega}\varphi.	
 \end{aligned}
 \end{equation}
 Now for the first term on the right hand side of the previous we consider $\delta>\overline{s}: \ \delta\not\in\{\eta: |\{u=\eta \}|>0\}$ and one has that
 $$\int_{\Omega}  h(u_p)f u_p^\frac{\gamma-1+p}{p} \varphi = \int_{\{u_p\le \delta\}}  h(u_p)f u_p^\frac{\gamma-1+p}{p} \varphi + \int_{\{u_p> \delta\}}  h(u_p)f u_p^\frac{\gamma-1+p}{p}  \varphi.$$
 We can pass to the limit first in $p\to 1$ and then as $\delta \to 0$ in the first term of the right hand side of the previous deducing that it goes to zero. For the second term, since $h(u_p)u_p^\frac{\gamma-1+p}{p}\chi_{\{u> \delta\}}$ is bounded in $L^\frac{m}{m-1}(\Omega)$, one can pass to the limit in $p$ by weak convergence and in $\delta$  by the Lebesgue Theorem. Hence one can take $p\to 1$ in \eqref{stimak>2} deducing
  \begin{equation}\label{stimak>3}	
 \begin{aligned}
 \int_{\Omega} \varphi|D u^{\gamma}| &+ \int_{\Omega} u^\gamma z \cdot \nabla \varphi + \int_\Omega g(u)u^\gamma \le \int_{\Omega}  h(u)f u^\gamma \varphi.	
 \end{aligned}
 \end{equation}
 Now as in Theorem \ref{teo_p>1} one can show that (observe that by Lemma \ref{lempairing} one has $(u^\gamma)^* \in L^1_{\rm loc}(\Omega,\operatorname{div}z)$)
 \begin{equation}\label{peru>1}
 -(u^\gamma)^* \operatorname{div}z +g(u)u^\gamma = h(u)fu^\gamma \ \text{ as measures in $\Omega$, }
 \end{equation}
 which, coupled with \eqref{stimak>3}, gives that 
 \begin{equation*}
 \int_{\Omega} \varphi|D u^{\gamma}| \le - \int_{\Omega} u^\gamma z \cdot \nabla \varphi -\int_{\Omega}(u^{\gamma})^* \varphi\operatorname{div}z = \int_{\Omega}\varphi (z, D u^\gamma),  \ \ \ \forall \varphi\in C^1_c(\Omega), \ \ \varphi \ge 0.	
\end{equation*}
Therefore, since the reverse inequality is trivial, one has
 \begin{equation*}\label{z>1}
(z, D u^\gamma)=|D u^{\gamma}|  \ \ \ \text{as measures in } \Omega.
 \end{equation*}
 At this point we apply Proposition $4.5$ of \cite{CDC} which gives
 \begin{equation*}
 \lambda(z,Du,x)=\lambda(z,Du^\gamma,x)  \qquad \text{for \(|Du|\)-a.e.}\ x\in\Omega, 
 \end{equation*}
 where \(\lambda(z, Du, \cdot)\) denotes the Radon-Nikodym derivative of $(z, Du)$ with respect to $|Du|$ and \(\lambda(z, Du^\lambda, \cdot)\) denotes the Radon-Nikodym derivative of $(z, Du^\gamma)$ with respect to $|Du^\gamma|$. This gives \eqref{def_zp=1strong}. Finally in order to deduce \eqref{def_bordop=1strong} we take $T^\gamma_k(u_n)$ as a test function in the weak formulation of \eqref{pbpn2} and applying the Young inequality one deduces
 \begin{equation*}
 \gamma^{\frac{1}{p}}\left(\frac{p}{\gamma-1+p}\right)\int_\Omega |\nabla T_k^{\frac{\gamma-1+p}{p}}(u_n)|+ \int_{\partial \Omega}T^\frac{\gamma-1+p}{p}_k(u_n) d\mathcal{H}^{N-1} + \int_\Omega g(u_n)T^\gamma_k(u_n) \le \int_\Omega h_n(u_n)f_nT^\gamma_k(u_n) + \frac{p-1}{p}|\Omega|,
 \end{equation*}
 and simply taking first $n\to \infty$ and then $p\to 1$ one gets
  \begin{equation}\label{bordo>1}
 \int_\Omega |D T_k^{\gamma}(u)| + \int_{\partial \Omega}T_k^\gamma(u) d\mathcal{H}^{N-1} + \int_\Omega g(u)T^\gamma_k(u) \le \int_\Omega h(u)fT^\gamma_k(u).
 \end{equation}
 Now we observe that, reasoning as to obtain \eqref{peru>1}, one can deduce that
  \begin{equation*}
 -(T^\gamma_k(u))^* \operatorname{div}z +g(u)T^\gamma_k(u) = h(u)fT^\gamma_k(u) \ \text{ as measures in $\Omega$, }
 \end{equation*}
 which gathered in \eqref{bordo>1} takes to 
   \begin{equation}\label{bordo2}
 \int_\Omega |D T_k^{\gamma}(u)| + \int_{\partial \Omega}T^\gamma_k(u) d\mathcal{H}^{N-1} \le -\int_\Omega (T^\gamma_k(u))^* \operatorname{div}z= \int_{\Omega} (z,DT_k^\gamma(u)) - \int_{\partial \Omega} [T_k^\gamma(u)z,\nu]d\mathcal{H}^{N-1}.
 \end{equation}
 Now reasoning as to prove \eqref{troncata} one can deduce  $(z, D T_k^\gamma(u))=|D T_k^{\gamma}(u)|$  as measures in $\Omega$, which used in \eqref{bordo2} gives \eqref{def_bordop=1strong}.
 This concludes the proof.	   
\end{proof}

\end{document}